\documentclass[12pt,oneside,reqno]{amsart}

\usepackage{amsmath,amssymb,amsthm,textcomp}
\usepackage{dsfont}
\usepackage{amsfonts,graphicx}
\usepackage[mathscr]{eucal}
\usepackage{color}
\usepackage{csquotes}
\usepackage{relsize}
\usepackage{hyperref}
\usepackage[backend=bibtex,%
firstinits=true,%
doi=false,%
isbn=true,%
url=false,%
maxnames=99]{biblatex}%

\AtEveryBibitem{\clearfield{issn}}
\AtEveryCitekey{\clearfield{issn}}
\numberwithin{equation}{section}
\DeclareNameAlias{sortname}{last-first}
\theoremstyle{definition}
\usepackage{mathtools}
\numberwithin{equation}{section}

\newcommand{\ncom}{\newcommand}

\ncom{\beq}{\begin{equation}}
	\ncom{\eeq}{\end{equation}}
\ncom{\bea}{\begin{eqnarray*}}
	\ncom{\eea}{\end{eqnarray*}}
\ncom{\beqa}{\begin{eqnarray}}
	\ncom{\eeqa}{\end{eqnarray}}
\ncom{\nno}{\nonumber}
\ncom{\non}{\nonumber}
\ncom{\ds}{\displaystyle}
\ncom{\half}{\frac{1}{2}}
\ncom{\mbx}{\makebox{.25cm}}
\ncom{\hs}{\mbox{\hspace{.25cm}}}
\ncom{\rar}{\rightarrow}
\ncom{\Rar}{\Rightarrow}
\ncom{\noin}{\noindent}
\ncom{\bc}{\begin{center}}
	\ncom{\ec}{\end{center}}
\ncom{\sz}{\scriptsize}
\ncom{\rf}{\ref}
\ncom{\s}{\sqrt{2}}
\ncom{\sgm}{\sigma}
\ncom{\Sgm}{\Sigma}
\ncom{\psgm}{\sigma^{\prime}}
\ncom{\dt}{\delta}
\ncom{\Dt}{\Delta}
\ncom{\lmd}{\lambda}
\ncom{\Lmd}{\Lambda}
\ncom{\Th}{\Theta}
\ncom{\e}{\eta}
\ncom{\eps}{\epsilon}
\ncom{\pcc}{\stackrel{P}{>}}
\ncom{\lp}{\stackrel{L_{p}}{>}}
\ncom{\dist}{{\rm\,dist}}
\ncom{\sspan}{{\rm\,span}}
\ncom{\re}{{\rm Re\,}}
\ncom{\im}{{\rm Im\,}}
\ncom{\sgn}{{\rm sgn\,}}
\ncom{\ba}{\begin{array}}
	\ncom{\ea}{\end{array}}
\ncom{\hone}{\mbox{\hspace{1em}}}
\ncom{\htwo}{\mbox{\hspace{2em}}}
\ncom{\hthree}{\mbox{\hspace{3em}}}
\ncom{\hfour}{\mbox{\hspace{4em}}}
\ncom{\vone}{\vskip 2ex}
\ncom{\vtwo}{\vskip 4ex}
\ncom{\vonee}{\vskip 1.5ex}
\ncom{\vthree}{\vskip 6ex}
\ncom{\vfour}{\vspace*{8ex}}
\ncom{\norm}{\|\;\;\|}
\ncom{\integ}[4]{\int_{#1}^{#2}\,{#3}\,d{#4}}
\ncom{\vspan}[1]{{{\rm\,span}\{ #1 \}}}
\ncom{\dm}[1]{ {\displaystyle{#1} } }
\ncom{\ri}[1]{{#1} \index{#1}}

\newtheorem{theorem}{\bf Theorem}[section]
\newtheorem{remark}{\bf Remark}[section]
\newtheorem{proposition}{Proposition}[section]
\newtheorem{lemma}{Lemma}[section]

\newtheoremstyle
{remarkstyle}
{}
{11pt}
{}
{}
{\bfseries}
{:}
{     }
{\thmname{#1} \thmnumber{#2} }

\theoremstyle{remarkstyle}

\def\eps{\varepsilon}

\begin{document}
\title{On Mixed Time-Changed Erlang Queue}
		\author[Rohini Bhagwanrao Pote]{Rohini Bhagwanrao Pote}
		\address{Rohini Bhagwanrao pote, Department of Mathematics, Indian Institute of Technology Bhilai, Durg 491002, India.}
		\email{rohinib@iitbhilai.ac.in}
		\author[Kuldeep Kumar Kataria]{Kuldeep Kumar Kataria}
		\address{Kuldeep Kumar Kataria, Department of Mathematics, Indian Institute of Technology Bhilai, Durg 491002, India.}
		\email{kuldeepk@iitbhilai.ac.in}
		\date{\today}
\begin{abstract}
	We study a time-changed variant of the Erlang queue by taking the first hitting time of a mixed stable subordinator as the time-changing component. We call it the mixed time-changed Erlang queue. We derive the system of fractional differential equations that governs its state probabilities. The explicit expressions for the state probabilities of mixed time-changed Erlang queue and their Laplace transform are derived. 
	Equivalently, it is represented in terms of phases and its mean queue length is obtained. Also, some distributional properties of the mixed time-changed Erlang queue such as the distribution of its inter-arrival times, inter-phase times, service times and busy period are derived. Later, its conditional waiting time is discussed and two plots of sample paths simulation are presented.


\end{abstract}
\maketitle 
\section{Introduction}
The Erlang queue is a queueing system in which the arrival of customers is modeled according to a Poisson process with rate $\lambda$ and its service system has Erlang distribution with shape parameter $k$ and mean $1/\mu$. Its inter-arrival times are exponentially distributed with parameter $\lambda$ and its service system consists of $k$ phases, where the service time of each phase is exponentially distributed with parameter $k\mu$.
In the fields of finance and telecommunication, the Erlang queues are quite useful. In call centers, the Erlang queue is used to predict number of agents required to handle incoming calls (see Spath and F{\"a}hnrich (2006)). Cahoy {\it et al.} (2015) used it to model financial data.

In Luchak (1956), (1958) a single channel queue with Poisson arrivals and a general
class of service-time distributions is studied. The transient solution of Erlang queue is obtained in Griffiths {\it et al.} (2006).
Cahoy {\it et al.} (2015) proposed the first fractional generalization of $M/M/1$ queue where they obtained its state probabilities and an algorithm to simulate fractional $M/M/1$ queue.
Di Crescenzo {\it et al.} (2003) studied $M/M/1$ queue with catastrophes and its fractional variant is discussed in Ascione {\it et al.} (2018). Giorno {\it et al.} (2018) studied a single-server queueing system with Poisson arrivals and state-dependent service mechanism characterized by logarithmic steady-state distribution.
The Erlang queue time-changed with inverse stable subordinator is studied by Ascione {\it et al.} (2020) where its state probabilities, mean queue length along with the distributions its of busy period, inter-arrival, inter-phase and service times are derived.

Let $\mathbb{N}$ denote the set of positive integers.
The Erlang queue $\{\mathcal{Q}(t)\}_{t\geq0}$ with state space $\mathcal{H}_{0}=\mathcal{H} \cup \{(0,0)\}$, 
where $\mathcal{H}=\{(n,s) \in \mathbb{N} \times \mathbb{N}:  s \leq k \} $ can be described as follows: $\mathcal{Q}(t)=(\mathcal{N}(t),\mathcal{S}(t))$. Here, $\mathcal{N}(t)$ denotes the number of customers in the system at time $t \geq 0$ and $\mathcal{S}(t)$ is the phase of customer being served at time $t$ provided $\mathcal{N}(t)>0$. 
For $\mathcal{N}(t)=0$, we take $\mathcal{S}(t)=0$.

For $t \geq 0$, let us denote the transient state probabilities of Erlang queue as follows:
$p_{0}(t)=\mathrm{Pr}(\mathcal{Q}(t)=(0,0)|\mathcal{Q}(0)=(0,0))$
and
$p_{n,s}(t)=\mathrm{Pr}(\mathcal{Q}(t)=(n,s)|\mathcal{Q}(0)=(0,0))$, $(n,s) \in \mathcal{H}$.
That is, $p_{0}(t)$ is the probability that there is no customers in the system at time $t$ and $p_{n,s}(t)$ is the probability that there are
$n$ customers in the system at time $t$ and the customer being served is at phase $s$. 

Equivalently,
the Erlang queue can be represented in terms of phase count as a queue length process (see Ascione {\it{et al.}} (2020), p. 3252). Consider a bijective map $g_{k}:\mathcal{H}_{0}\rightarrow\mathbb{N}_{0}$ defined as follows:  
\begin{equation}\label{gk12}
	g_{k}(n,s)\coloneqq \left\{
	\begin{array}{ll}
		k(n-1)+s,\, (n,s)\in\mathcal{H},\\
		0,\, (n,s)=(0,0),
	\end{array}
	\right.  
\end{equation}
where $\mathbb{N}_{0}=\mathbb{N}\cup\{0\}$.
Its inverse map $(a_{k}(m),b_{k}(m))$ is such that
\begin{equation}\label{bk12}
	b_{k}(m)=\left\{
	\begin{array}{ll}
		\min\{s>0: s \equiv m \pmod{k} \},\, m>0,\\
		0,\, m=0,
	\end{array}
	\right.  
\end{equation}
and
\begin{equation}\label{ak12}
	a_{k}(m)=\left\{
	\begin{array}{ll}
		\frac{m-b_{k}(m)}{k}+1,\, m>0,\\
		0,\, m=0.
	\end{array}
	\right.  
\end{equation}
Then, $\mathcal{L}(t)=g_{k}(\mathcal{Q}(t))$ is the length of Erlang queue in terms of number of phases at time $t\geq0$. Its state probabilities are denoted by $p_{n}(t)=\mathrm{Pr}(\mathcal{L}(t)=n|\mathcal{L}(0)=0)$, $n\geq0$.
Note that $p_{n}(t)=p_{a_{k}(n),b_{k}(n)}(t)$.

It is well known that the occurrence of catastrophic events such as earthquakes, tsunami, {\it etc.} exhibits long memory. It is empirically proven that these real life situations can be modeled by time-changed processes as they exhibits long range behaviour. Recently, the time-changed processes are extensively studied by many researchers, for example, the time fractional Poisson process (see Beghin and Orsingher (2009), Meerschaert {\it et al.} (2011)), space fractional Poisson process (see Orsingher and Polito (2012)), time-changed birth-death processes (see Orsingher and Polito (2011), Kataria and Vishwakarma (2025)) {\it etc}. Beghin (2012) studied random-time processes governed by differential equations of fractional distributed order, Kataria and Khandakar (2021) studied mixed fractional risk process which are based on the inverse mixed stable subordinator.


In this paper, we introduce and study a time-changed variant of the Erlang queue where the time changing component used is the first hitting time of a mixed stable subordinator. We call it the mixed time-changed Erlang queue. The paper is organized as follows:

First, we give some preliminary results and known definitions of special functions, fractional derivative, subordinators and their inverse process, {\it etc.} 
Then, we define the mixed time-changed Erlang queue by time-changing it with inverse mixed stable subordinator which is the first hitting time of a mixed stable subordinator. The system of fractional differential equations that governs its state probabilities is derived. The Laplace transform of its state probabilities is obtained whose inversion yields its distribution. Also, the mixed time-changed Erlang queue is represented as a queue length process which characterizes it in terms of number of phases. The fractional differential equation that governs its mean queue length is derived, and an explicit expression for the mean queue length is obtained. Moreover, we derive the distribution of inter-arrival times, inter-phase times, service times and busy period of mixed time-changed Erlang queue, and discussed its conditional waiting time. Some plots of sample paths simulation are given.

\section{Preliminaries}\label{sec2}
Here, we collect some known definitions and results that will be used.
\subsection{Mittag-Leffler function}
The three-parameter Mittag-Leffler function is defined as (see Kilbas {\it et al.} (2006))
\begin{equation}\label{Mittag12}
	E_{\alpha,\beta}^{\gamma}(t)\coloneqq\sum_{r=0}^{\infty}\frac{\Gamma(\gamma+r)t^{r}}{r!\Gamma(\gamma)\Gamma(\alpha r+\beta)},\,t\in\mathbb{R},\,\alpha,\,\beta,\,\gamma>0,
\end{equation}
where $\mathbb{R}$ denote set of real numbers.
Its Laplace transform is given by (see Haubold {\it et al.} (2011))
\begin{equation}\label{ltm12}
	\mathbb{L}(t^{\beta-1}E_{\alpha,\beta}^{\gamma}(\omega t^{\alpha}))(z)=\frac{z^{\alpha\gamma-\beta}}{(z^{\alpha}-\omega)^{\gamma}},\, \omega\in\mathbb{R},\,\,z>0,\,|\omega z^{\alpha}|<1.
\end{equation}

The following result will be used (see Haubold {\it et al.} (2011), Eq. (17.6)):
\begin{equation}\label{rholp}
	\mathbb{L}^{-1}\Big(\frac{z^{\rho-1}}{z^{\alpha}+az^{\beta}+b};t\Big)=t^{\alpha-\rho}\sum_{h=0}^{\infty}(-a)^{h}t^{(\alpha-\beta)h}E^{h+1}_{\alpha,\alpha+(\alpha-\beta)h-\rho+1}(-bt^{\alpha}),
\end{equation}
where $\alpha>\beta>0$, $\alpha-\rho+1>0$ and $|az^{\beta}/(z^{\alpha}+b)|<1$.
\subsection{Fractional derivative}
The Caputo fractional derivative of a function $g(\cdot)$ is defined as follows (see Kilbas {\it{et al.}} (2006)): 
\begin{equation}\label{caputo}
	\frac{\mathrm{d}^{\alpha}}{\mathrm{d}t^{\alpha}}g(t)=
	\begin{cases}
		\frac{1}{\Gamma(1-\alpha)}\int_{0}^{t}(t-y)^{-\alpha}g'(y)\mathrm{d}y,\, 0<\alpha<1, \\
		g'(y), \, \alpha=1.
	\end{cases}
\end{equation}
Its Laplace transform is given by
\begin{equation}\label{caputolp}
	\mathbb{L}\Big(\frac{\mathrm{d}^{\alpha}}{\mathrm{d}t^{\alpha}}g(t)\Big)(z)=z^{\alpha}\Tilde{g}(z)-z^{\alpha-1}g(0),
\end{equation} 
where $\tilde{g}(z)$ denotes the Laplace transform of a function $g(t)$.
\subsection{Stable subordinators and their inverse}
An $\alpha$-stable subordinator is a one dimensional L\'evy process $\{D_{\alpha}(t)\}_{t\geq0}$, $0<\alpha<1$ whose Laplace transform is given by $\mathbb{E}(e^{-zD_{\alpha}(t)})=e^{-tz^{\alpha}}$, $z>0$ (see Applebaum (2009)).
Its first hitting time process $\{Y_{\alpha}(t)\}_{t\geq0}$, where $Y_{\alpha}(t)=\inf\{s\geq0:D_{\alpha}(s)>t\}$ is known as the inverse $\alpha$-stable subordinator. The Laplace transform of its probability density function is given by (see Meerschaert {\it{et al.}} (2013))
\begin{equation*}\label{ltLnu12}
	\mathbb{L}(\mathrm{Pr}\{Y_{\alpha}(t)\in\mathrm{d}y\})(z)=z^{\alpha-1}e^{-yz^{\alpha}}\mathrm{d}y, \, z>0.
\end{equation*}

The mixed stable subordinator $\{D_{\alpha_{1},\alpha_{2}}(t)\}_{t\geq0}$ is a L\'evy process whose Laplace transform is given by 
\begin{equation}\label{mixdsub12}
	\mathbb{E}(e^{-zD_{\alpha_{1},\alpha_{2}}(t)})=e^{-t(c_{1}z^{\alpha_{1}}+c_{2}z^{\alpha_{2}})},\, z>0,
\end{equation}
where $c_{1}\geq0$, $c_{2}\geq0$, $c_{1}+c_{2}=1$, $0<\alpha_{2}<\alpha_{1}<1$. 
Its first hitting time process $\{Y_{\alpha_{1},\alpha_{2}}(t)\}_{t\geq0}$ is known as the inverse mixed stable subordinator, where
\begin{equation*}\label{invmixdsub12}
	Y_{\alpha_{1},\alpha_{2}}(t)=\text{inf}\{s\geq0:D_{\alpha_{1},\alpha_{2}}(s)>t\},\, t>0
\end{equation*}
with $Y_{\alpha_{1},\alpha_{2}}(0)=0$. For $c_{1},c_{2}\in\{0,1\}$, it reduces to inverse stable subordinator. The probability density function $f_{\alpha_{1},\alpha_{2}}(t,u)=\frac{\mathrm{d}}{\mathrm{d}u}\mathrm{Pr}(Y_{\alpha_{1},\alpha_{2}}(t)\leq u)$ of inverse mixed stable subordinator has the following Laplace transform:
\begin{equation}{\label{mixedlp12}}
	\int_{0}^{\infty}e^{-zt}f^{\alpha_{1},\alpha_{2}}(t,u)dt=\frac{c_{1}z^{\alpha_{1}}+c_{2}z^{\alpha_{2}}}{z}e^{-u(c_{1}z^{\alpha_{1}}+c_{2}z^{\alpha_{2}})},\,z>0.
\end{equation}
\section{Mixed Time-Changed Erlang Queue}\label{sec3}
In this section, we introduce a time-changed Erlang queue where the time changing component is the first hitting time of mixed stable subordinator. We call it the mixed time-changed Erlang queue. 

Let $\{\mathcal{Q}(t)\}_{t\geq0}$ be an Erlang queue, that is, $\mathcal{Q}(t)=(\mathcal{N}(t),\mathcal{S}(t))$, $t\geq0$, where $\mathcal{N}(t)$ denotes the number of customers in the system at time $t$ and $\mathcal{S}(t)$ denotes the phase of customer being served at time $t$. We define the mixed time-changed Erlang queue $\{\mathcal{Q}^{\alpha_{1},\alpha_{2}}(t)\}_{t\geq0}$, $0<\alpha_{2}<\alpha_{1}<1$ as follows: 
\begin{equation*}
	\mathcal{Q}^{\alpha_{1},\alpha_{2}}(t)\coloneqq\mathcal{Q}(Y_{\alpha_{1},\alpha_{2}}(t)),
\end{equation*}
where the Erlang queue $\{\mathcal{Q}(t)\}_{t\geq0}$ is independent of the inverse mixed stable subordinator $\{Y_{\alpha_{1},\alpha_{2}}(t)\}_{t\geq0}$.
For $t\geq0$, let us denote its state probabilities by
\begin{align*}
	p_{0}^{\alpha_{1},\alpha_{2}}(t)&=\mathrm{Pr}(\mathcal{Q}^{\alpha_{1},\alpha_{2}}(t)=(0,0)|\mathcal{Q}^{\alpha_{1},\alpha_{2}}(0)=(0,0)),\\
	p_{n,s}^{\alpha_{1},\alpha_{2}}(t)&=\mathrm{Pr}(\mathcal{Q}^{\alpha_{1},\alpha_{2}}(t)=(n,s)|\mathcal{Q}^{\alpha_{1},\alpha_{2}}(0)=(0,0)),\, (n,s)\in \mathcal{H},
\end{align*}
where $\mathcal{H}=\{(n,s) \in \mathbb{N} \times \mathbb{N}:  s \leq k \}$.	 
\begin{remark}
	For $c_{1}=1$ or $c_{2}=1$ in \eqref{mixdsub12}, the mixed time-changed Erlang queue reduces to the fractional Erlang queue (see Ascione {\it et al.} (2020)).
\end{remark} 
\begin{theorem}\label{DE12}
	 Let $c_{1}\geq0$, $c_{2}\geq0$ such that $c_{1}+c_{2}=1$ and $\frac{\mathrm{d}^{\alpha}}{\mathrm{d}t^{\alpha}}$ be the Caputo fractional derivative of order $\alpha$ as defined in \eqref{caputo}. Then,
	the state probabilities of mixed time-changed Erlang queue solves the following system of differential equations: 
\begin{align*}
	\Big(c_{1}\frac{\mathrm{d}^{\alpha_{1}}}{\mathrm{d}t^{\alpha_{1}}}+c_{2}\frac{\mathrm{d}^{\alpha_{2}}}{\mathrm{d}t^{\alpha_{2}}}\Big)p_{0}^{\alpha_{1},\alpha_{2}}(t)
	&=-\lambda p_{0}^{\alpha_{1},\alpha_{2}}(t)+k\mu p_{1,1}^{\alpha_{1},\alpha_{2}}(t),  \\
	\Big(c_{1}\frac{\mathrm{d}^{\alpha_{1}}}{\mathrm{d}t^{\alpha_{1}}}+c_{2}\frac{\mathrm{d}^{\alpha_{2}}}{\mathrm{d}t^{\alpha_{2}}}\Big)p_{1,s}^{\alpha_{1},\alpha_{2}}(t)
	&=-(\lambda + k \mu)p_{1,s}^{\alpha_{1},\alpha_{2}}(t)+k\mu p_{1,s+1}^{\alpha_{1},\alpha_{2}}(t), \, 1 \leq s \leq k-1, \\
	\Big(c_{1}\frac{\mathrm{d}^{\alpha_{1}}}{\mathrm{d}t^{\alpha_{1}}}+c_{2}\frac{\mathrm{d}^{\alpha_{2}}}{\mathrm{d}t^{\alpha_{2}}}\Big)p_{1,k}^{\alpha_{1},\alpha_{2}}(t)
	&=-(\lambda + k \mu)p_{1,k}^{\alpha_{1},\alpha_{2}}(t)+k\mu p_{2,1}^{\alpha_{1},\alpha_{2}}(t)+ \lambda p_{0}^{\alpha_{1},\alpha_{2}}(t),  \\
	\Big(c_{1}\frac{\mathrm{d}^{\alpha_{1}}}{\mathrm{d}t^{\alpha_{1}}}+c_{2}\frac{\mathrm{d}^{\alpha_{2}}}{\mathrm{d}t^{\alpha_{2}}}\Big)p_{n,s}^{\alpha_{1},\alpha_{2}}(t)
	&=-(\lambda + k \mu)p_{n,s}^{\alpha_{1},\alpha_{2}}(t)+k\mu p_{n,s+1}^{\alpha_{1},\alpha_{2}}(t)+\lambda p_{n-1,s}^{\alpha_{1},\alpha_{2}}(t), \\
	&\hspace{6.1cm}  1\leq s \leq k-1,  \, n \geq 2,  \\
	\Big(c_{1}\frac{\mathrm{d}^{\alpha_{1}}}{\mathrm{d}t^{\alpha_{1}}}+c_{2}\frac{\mathrm{d}^{\alpha_{2}}}{\mathrm{d}t^{\alpha_{2}}}\Big)p_{n,k}^{\alpha_{1},\alpha_{2}}(t)
	&=-(\lambda + k \mu)p_{n,k}^{\alpha_{1},\alpha_{2}}(t)+k\mu p_{n+1,1}^{\alpha_{1},\alpha_{2}}(t)+\lambda p_{n-1,k}^{\alpha_{1},\alpha_{2}}(t),\, n\geq2,
	\end{align*}
with $p_{0}^{\alpha_{1},\alpha_{2}}(0)=1$ and
$p_{n,s}^{\alpha_{1},\alpha_{2}}(0)=0$ for $1\leq s\leq k$, $n\geq 1$.
\end{theorem}
\begin{proof}				
For $(n,s) \in \mathcal{H}$, we have
\begin{align}\label{pnsalphat12}
	p_{n,s}^{\alpha_{1},\alpha_{2}}(t)
	&=\mathrm{Pr}(\mathcal{Q}(Y_{\alpha_{1},\alpha_{2}}(t))=(n,s)|\mathcal{Q}(0)=(0,0))\nonumber\\
	&=\int_{0}^{\infty}\mathrm{Pr}(\mathcal{Q}(y)=(n,s)|\mathcal{Q}(0)=(0,0))\mathrm{Pr}\{Y_{\alpha_{1},\alpha_{2}}(t)\in\mathrm{d}y\}\nonumber\\
	&=\int_{0}^{\infty}p_{n,s}(y)\mathrm{Pr}\{Y_{\alpha_{1},\alpha_{2}}(t)\in\mathrm{d}y\},
\end{align}
and
\begin{equation}\label{p0alpha12}
	{p}_{0}^{\alpha_{1},\alpha_{2}}(t)=\int_{0}^{\infty}p_{0}(y)\mathrm{Pr}\{Y_{\alpha_{1},\alpha_{2}}(t)\in\mathrm{d}y\},
\end{equation}
where $p_{n,s}(y)$ and $p_{0}(y)$ are the state probability and zero state probability of Erlang queue, respectively.

Let the generating function of $\{\mathcal{Q}^{\alpha_{1},\alpha_{2}}(t)\}_{t\geq0}$ be given by
\begin{equation}\label{Galpha12}
	G^{\alpha_{1},\alpha_{2}}(x,t)=p_{0}^{\alpha_{1},\alpha_{2}}(t)+\sum_{n=1}^{\infty} \sum_{s=1}^{k}p_{n,s}^{\alpha_{1},\alpha_{2}}(t)x^{(n-1)k+s}, \,t\geq0,\,|x|\leq1.
\end{equation}
On substituting \eqref{pnsalphat12} and \eqref{p0alpha12} in \eqref{Galpha12}, we get
\begin{equation}\label{Gxtalpha12}
	{G}^{\alpha_{1},\alpha_{2}}(x,t)=\int_{0}^{\infty}G(x,y)\mathrm{Pr}\{Y_{\alpha_{1},\alpha_{2}}(t)\in\mathrm{d}y\},
\end{equation}
where $G(x,t)$ is the generating function of $\{\mathcal{Q}(t)\}_{t\geq0}$.
On taking the Laplace transform of \eqref{p0alpha12} and \eqref{Gxtalpha12}, and using \eqref{mixedlp12}, we get
\begin{equation}\label{p0alphalp12}
	\Tilde{p_{0}}^{\alpha_{1},\alpha_{2}}(z)=(c_{1}z^{\alpha_{1}-1}+c_{2}z^{\alpha_{2}-1})\int_{0}^{\infty}p_{0}(y)e^{-y({c_{1}z^{\alpha_{1}}+c_{2}z^{\alpha_{2}}})}\mathrm{d}y
\end{equation}
and
\begin{equation}\label{Gxtalphalp12}
	\Tilde{G}^{\alpha_{1},\alpha_{2}}(x,z)=(c_{1}z^{\alpha_{1}-1}+c_{2}z^{\alpha_{2}-1})\int_{0}^{\infty}G(x,y)e^{-y(c_{1}z^{\alpha_{1}}+c_{2}z^{\alpha_{2}})}\mathrm{d}y.
\end{equation}
	
On multiplying the first equation by \(x\), the second by \(x^{s+1}\), the third by \(x^{k+1}\), the fourth by \(x^{k(n-1)+s+1}\) and the fifth by \(x^{nk+1}\) in the system of differential equations given in Theorem \ref{DE12}, and then summing these equations over the range of \(n\) and \(s\), it can be established that the state probabilities of $\{\mathcal{Q}^{\alpha_{1},\alpha_{2}}(t)\}_{t\geq0}$ are solution of the system of differential equations given in Theorem \ref{DE12} if and only if its generating function is the solution of following equation:
{\small\begin{equation}\label{Gcauchy12}
	x\Big(c_{1}\frac{\mathrm{d}^{\alpha_{1}}}{\mathrm{d}t^{\alpha_{1}}}+c_{2}\frac{\mathrm{d}^{\alpha_{2}}}{\mathrm{d}t^{\alpha_{2}}}\Big)G^{\alpha_{1},\alpha_{2}}(x,t)=(1-x)\big((k\mu-\lambda (x+\dots+x^{k}))G^{\alpha_{1},\alpha_{2}}(x,t)-k\mu p_{0}^{\alpha_{1},\alpha_{2}}(t)\big),
\end{equation}}
with initial condition $G^{\alpha_{1},\alpha_{2}}(x,0)=1$. 
On taking the Laplace transform on both sides of \eqref{Gcauchy12} and using \eqref{caputolp}, we obtain 
\begin{equation}\label{cauchyLL12}
	x(c_{1}z^{\alpha_{1}-1}+c_{2}z^{\alpha_{2}-1})(z{\tilde{G}}^{\alpha_{1},\alpha_{2}}(x,z)-1)=(1-x)\big((k\mu-\lambda(x+\dots+x^{k}))\tilde{G}^{\alpha_{1},\alpha_{2}}(x,z)-k\mu\tilde{p_{0}}^{\alpha_{1},\alpha_{2}}(z)\big).
\end{equation}
By using \eqref{p0alphalp12} and \eqref{Gxtalphalp12} in \eqref{cauchyLL12}, we have 
\begin{align}\label{lptcde}
	x\Big((c_{1}z^{\alpha_{1}}&+c_{2}z^{\alpha_{2}})
	\int_{0}^{\infty}G(x,y)e^{-y(c_{1}z^{\alpha_{1}}+c_{2}z^{\alpha_{2}})}\mathrm{d}y-1\Big)\nonumber\\
	=&\int_{0}^{\infty}(1-x)\big((k\mu-\lambda(x+\dots+x^{k}))G(x,y)-k\mu p_{0}(y)\big)e^{-y(c_{1}z^{\alpha_{1}}+c_{2}z^{\alpha_{2}})}\mathrm{d}y. 
\end{align}
Now, on substituting Eq. (6) of Griffiths {\it et al.} (2006) in \eqref{lptcde}, we get the following identity:
\begin{equation*}x\Big((c_{1}z^{\alpha_{1}}+c_{2}z^{\alpha_{2}})\int_{0}^{\infty}G(x,y)e^{-y(c_{1}z^{\alpha_{1}}+c_{2}z^{\alpha_{2}})}\mathrm{d}y-1\Big)=
	x\int_{0}^{\infty}\frac{\partial }{\partial y}G(x,y)e^{-y(c_{1}z^{\alpha_{1}}+c_{2}z^{\alpha_{2}})}\mathrm{d}y.
\end{equation*}
This proves the required result.
\end{proof}
\begin{remark}
	For $c_{1}=1$ or $c_{2}=1$ in Theorem \eqref{DE12}, it reduces to the system of differential equations that governs the state probabilities of fractional Erlang queue (see Ascione \textit{et al.} (2020), Eq. (18)).
\end{remark}
The following result will be used:
\begin{proposition}\label{Lpp012}
	The Laplace transform of zero state probability of $\{\mathcal{Q}^{\alpha_{1},\alpha_{2}}(t)\}_{t\geq0}$ is given by
\begin{equation}\label{p0lp121}
	\tilde{p}_{0}^{\alpha_{1},\alpha_{2}}(z)=\sum_{m=1}^{\infty}\sum_{r=0}^{\infty}A^{0}_{m,r}\frac{(c_{1}z^{\alpha_{1}-1}+c_{2}z^{\alpha_{2}-1})}{(\lambda+k\mu+c_{1}z^{\alpha_{1}}+c_{2}z^{\alpha_{2}})^{a_{m,r}^{0}}},
\end{equation} 
where
\begin{equation}\label{amr12}
	a_{m,r}^{0}=m+r(k+1)
\end{equation}	
and
\begin{equation}\label{Amr12}
	A^{0}_{m,r}=\frac{m\lambda^{r}(k\mu)^{m+rk-1}}{r!(m+rk)!}(m+r(k+1)-1)!.
\end{equation}
\end{proposition}	
\begin{proof}
By using Eq. (5) of Ascione {\it et al.} (2020) in \eqref{p0alphalp12}, we get 
{\small\begin{align}\label{p00lp121}
	\Tilde{{p}_{0}}^{\alpha_{1},\alpha_{2}}(z)
	&=\sum_{m=1}^{\infty}\sum_{r=0}^{\infty}\frac{m\lambda^{r}(k\mu)^{m+rk-1}}{r!(m+rk)!}(c_{1}z^{\alpha_{1}-1}+c_{2}z^{\alpha_{2}-1})\int_{0}^{\infty}y^{m+r(k+1)-1}e^{-(\lambda+k\mu+c_{1}z^{\alpha_{1}}+c_{2}z^{\alpha_{2}})y}\mathrm{d}y\nonumber\\		&=\sum_{m=1}^{\infty}\sum_{r=0}^{\infty}\frac{m\lambda^{r}(k\mu)^{m+rk-1}}{r!(m+rk)!}\frac{(c_{1}z^{\alpha_{1}-1}+c_{2}z^{\alpha_{2}-1})(m+r(k+1)-1)!}{(\lambda+k\mu+c_{1}z^{\alpha_{1}}+c_{2}z^{\alpha_{2}})^{m+r(k+1)}}.
\end{align}}	
On substituting \eqref{amr12} and \eqref{Amr12} in \eqref{p00lp121}, we get the required result.
\end{proof}
\begin{remark}
	For $c_{1}=1$ or $c_{2}=1$ in \eqref{p0lp121}, it reduces to the Laplace transform of fractional Erlang queue (see Ascione {\it et al.} (2020), Eq. (37)).
\end{remark}
\begin{theorem}\label{thmp012}
	Let $f^{*(n)}$ denote the $n$-fold convolution of any function $f(\cdot)$ and $E_{\alpha,\beta}^{\gamma}(\cdot)$ be the three parameter Mittag-Leffler function defined in \eqref{Mittag12}. Then,
	the zero state probability of $\{\mathcal{Q}^{\alpha_{1},\alpha_{2}}(t)\}_{t\geq0}$ is given by
\begin{equation}\label{finp0alpha12}
	p_{0}^{\alpha_{1},\alpha_{2}}(t)=\sum_{m=1}^{\infty}\sum_{r=0}^{\infty}\frac{A^{0}_{m,r}}{c_{1}^{a_{m,r}^{0}}}\Big(c_{1}f_{a_{m,r}^{0}}^{*(a_{m,r}^{0})}(t)+c_{2}g_{a_{m,r}^{0}}^{*(a_{m,r}^{0})}(t)\Big),
\end{equation}
where $a_{m,r}^{0}$ and $A_{m,r}^{0}$ are given in \eqref{amr12} and \eqref{Amr12}, respectively. Here,
\begin{equation}\label{f12}
	f_{N}(t)=t^{\alpha_{1}-\big(\frac{\alpha_{1}-1}{N}\big)-1}\sum_{h=0}^{\infty}\Big(-\frac{c_{2}}{c_{1}}\Big)^{h}t^{(\alpha_{1}-\alpha_{2})h}E^{h+1}_{\alpha_{1},\alpha_{1}+(\alpha_{1}-\alpha_{2})h+\frac{1-\alpha_{1}}{N}}\Big(-\Big(\frac{\lambda+k\mu}{c_{1}}\Big)t^{\alpha_{1}}\Big)
\end{equation} and
\begin{equation}\label{g12}
	g_{N}(t)=t^{\alpha_{1}-\big(\frac{\alpha_{2}-1}{N}\big)-1}\sum_{h=0}^{\infty}\Big(-\frac{c_{2}}{c_{1}}\Big)^{h}t^{(\alpha_{1}-\alpha_{2})h}E^{h+1}_{\alpha_{1},\alpha_{1}+(\alpha_{1}-\alpha_{2})h+\frac{1-\alpha_{2}}{N}}\Big(-\Big(\frac{\lambda+k\mu}{c_{1}}\Big)t^{\alpha_{1}}\Big).
\end{equation}
\end{theorem}
\begin{proof}
On taking the inverse Laplace transform of \eqref{p0lp121}, we obtain
\begin{align}\label{p01212}
	p_{0}^{\alpha_{1},\alpha_{2}}(t)&=\sum_{m=1}^{\infty}\sum_{r=0}^{\infty}A^{0}_{m,r}\Bigg(\frac{1}{c_{1}^{a_{m,r}^{0}-1}}\mathbb{L}^{-1}\Bigg(\bigg(\frac{z^{\frac{\alpha_{1}-1}{a_{m,r}^{0}}}}{z^{\alpha_{1}}+\frac{c_{2}}{c_{1}}z^{\alpha_{2}}+\frac{\lambda+k\mu}{c_{1}}}\bigg)^{a_{m,r}^{0}};t\Bigg)\nonumber \\
	&\hspace{5cm} +\frac{c_{2}}{c_{1}^{a_{m,r}^{0}}}\mathbb{L}^{-1}\Bigg(\bigg(\frac{z^{\frac{\alpha_{2}-1}{a_{m,r}^{0}}}}{z^{\alpha_{1}}+\frac{c_{2}}{c_{1}}z^{\alpha_{2}}+\frac{\lambda+k\mu}{c_{1}}}\bigg)^{a_{m,r}^{0}};t\Bigg)\Bigg).
\end{align}
On using \eqref{rholp} in \eqref{p01212}, we get the required result.
\end{proof}
\begin{remark}\label{p0tremark12}
	In \eqref{finp0alpha12}, if $c_{1}=1$ then $c_{2}=0$. So,
\begin{equation}\label{p0tc112}
	p_{0}^{\alpha_{1},\alpha_{2}}(t)=\sum_{m=1}^{\infty}\sum_{r=0}^{\infty}A^{0}_{m,r}f_{a_{m,r}^{0}}^{*(a_{m,r}^{0})}(t)
\end{equation}
and 
\begin{equation}\label{p0r12}
	f_{a_{m,r}^{0}}(t)=t^{\alpha_{1}-\big(\frac{\alpha_{1}-1}{a_{m,r}^{0}}\big)-1}E^{1}_{\alpha_{1},\alpha_{1}+\frac{1-\alpha_{1}}{a_{m,r}^{0}}}(-(\lambda+k\mu)t^{\alpha_{1}}).
\end{equation}
On taking the Laplace transform of \eqref{p0r12} and using \eqref{ltm12}, we have
\begin{equation}\label{fr12}
	\tilde{f}_{a_{m,r}^{0}}(z)=\frac{z^{\frac{\alpha_{1}-1}{a_{m,r}^{0}}}}{z^{\alpha_{1}}+\lambda+k\mu}.
\end{equation}
Thus, the Laplace transform of $a_{0}^{m,r}$-fold  convolution of \eqref{fr12} is given by
\begin{equation}\label{frr12}
	\tilde{f}_{a_{m,r}^{0}}^{*(a_{m,r}^{0})}(z)=\frac{z^{\alpha_{1}-1}}{(z^{\alpha_{1}}+\lambda+k\mu)^{a_{m,r}^{0}}}.
\end{equation}
In \eqref{p0tc112}, on taking the inverse Laplace transform of \eqref{frr12} and using \eqref{ltm12}, we get the zero state probability of fractional Erlang queue (see Ascione {\it et al.} (2020), Eq. (36)). 
Similar result holds if $c_{2}=1$.
\end{remark}
In the subsequent results, we will be using the following notations:
\begin{align}\label{constnts12}
	\left.
	\begin{aligned}
		B_{i}^{n,s}&=\frac{\lambda^{n+i}(k\mu)^{k(i+1)-s}}{(k(i+1)-s)!(n+i)!}(n+k-s+i(k+1))!,\\
		\delta_{i}^{n,s}&=n-s+(i+1)(k+1),\\
		C_{i,m,r}^{n,s}&=k\mu B_{i}^{n,s}A_{m,r}^{0},\\
		\nu_{i,m,r}^{n,s}&=\delta_{i}^{n,s}+a_{m,r}^{0},\\
		D_{i,m,r}^{n,s}&=\begin{cases}
			k\mu B_{i}^{n,s+1}A_{m,r}^{0},\,s=1,2,\dots,k-1,\\
			k\mu B_{i}^{n+1,1}A_{m,r}^{0},\, s=k,
		\end{cases}	\\
		\pi_{i,m,r}^{n,s}&=\begin{cases}
			\delta_{i}^{n,s+1}+a_{m,r}^{0},\,s=1,2,\dots,k-1,\\
			\delta_{i}^{n+1,1}+a_{m,r}^{0},\, s=k,
		\end{cases}\\
	\end{aligned}
	\right\}
\end{align}
where $a_{m,r}^{0}$ and $A_{m,r}^{0}$ are as given in \eqref{amr12} and \eqref{Amr12}, respectively.
\begin{proposition}\label{proppns12}
	The Laplace transform of state probabilities of $\{\mathcal{Q}^{\alpha_{1},\alpha_{2}}(t)\}_{t\geq0}$ are given by
\begin{align}\label{pnslp12}
	\Tilde{p_{n,s}}^{\alpha_{1},\alpha_{2}}(z)&=(c_{1}z^{\alpha_{1}-1}+c_{2}z^{\alpha_{2}-1})\Big(\sum_{i=0}^{\infty}\frac{B_{i}^{n,s}}{(\lambda+k\mu+c_{1}z^{\alpha_{1}}+c_{2}z^{\alpha_{2}})^{\delta^{n,s}_{i}}}\nonumber\\
	&\hspace{1.2cm} +\sum_{i=0}^{\infty}\sum_{m=1}^{\infty}\sum_{r=0}^{\infty}\frac{C^{n,s}_{i,m,r}}{(\lambda+k\mu+c_{1}z^{\alpha_{1}}+c_{2}z^{\alpha_{2}})^{\nu_{i,m,r}^{n,s}}}\nonumber\\
	&\hspace{1.2cm} -\sum_{i=0}^{\infty}\sum_{m=1}^{\infty}\sum_{r=0}^{\infty}\frac{D^{n,s}_{i,m,r}}{(\lambda+k\mu+c_{1}z^{\alpha_{1}}+c_{2}z^{\alpha_{2}})^{\pi_{i,m,r}^{n,s}}}\Big),\,n\geq1,\,1\leq s\leq k.
\end{align}
\end{proposition}
\begin{proof}
	On taking the Laplace transform of \eqref{pnsalphat12} and using \eqref{mixedlp12}, we get
\begin{equation}\label{pnsalphalp12}
	\Tilde{p_{n,s}}^{\alpha_{1},\alpha_{2}}(z)=(c_{1}z^{\alpha_{1}-1}+c_{2}z^{\alpha_{2}-1})\int_{0}^{\infty}p_{n,s}(y)e^{-y({c_{1}z^{\alpha_{1}}+c_{2}z^{\alpha_{2}}})}\mathrm{d}y.
\end{equation}
For $n\geq1$ and $1\leq s\leq k-1$, by using Eq. (6) of Ascione {\it et al.} (2020) in \eqref{pnsalphalp12}, we obtain 
\begin{align}\label{pnslp121}
	\Tilde{{p}_{n,s}}^{\alpha_{1},\alpha_{2}}(z)
	&=(c_{1}z^{\alpha_{1}-1}+c_{2}z^{\alpha_{2}-1})\Big(\sum_{i=0}^{\infty}\frac{\lambda^{n+i}(k\mu)^{k(i+1)-s}}{(k(i+1)-s)!(n+i)!}\int_{0}^{\infty}y^{n+k-s+i(k+1)}\nonumber \\ 
	&\ \ \cdot e^{-(\lambda+k\mu+c_{1}z^{\alpha_{1}}+c_{2}z^{\alpha_{2}})y}\mathrm{d}y +\sum_{i=0}^{\infty}\frac{\lambda^{n+i}(k\mu)^{k(i+1)-s+1}}{(k(i+1)-s)!(n+i)!}\nonumber\\
	&\ \ \cdot\int_{0}^{\infty}\int_{0}^{y}p_{0}(u)(y-u)^{n+k-s+i(k+1)}e^{-(\lambda+k\mu)(y-u)}e^{-(c_{1}z^{\alpha_{1}}+c_{2}z^{\alpha_{2}})y}\mathrm{d}u\mathrm{d}y\nonumber\\
	&\ \ -\sum_{i=0}^{\infty}\frac{\lambda^{n+i}(k\mu)^{k(i+1)-s}}{(k(i+1)-s-1)!(n+i)!} \int_{0}^{\infty}\int_{0}^{y}p_{0}(u)(y-u)^{n+k-s-1+i(k+1)}\nonumber\\
	&\ \ \cdot e^{-(\lambda+k\mu)(y-u)} e^{-(c_{1}z^{\alpha_{1}}+c_{2}z^{\alpha_{2}})y}\mathrm{d}u\mathrm{d}y\Big).
\end{align}	
Also, by using Eq. (7) of Ascione {\it et al.} (2020) in \eqref{pnsalphalp12}, we have
\begin{align}\label{pnklp121}
	\Tilde{{p}_{n,k}}^{\alpha_{1},\alpha_{2}}(z)
	&=(c_{1}z^{\alpha_{1}-1}+c_{2}z^{\alpha_{2}-1})\Big(\sum_{i=0}^{\infty}\frac{\lambda^{n+i}(k\mu)^{ki}}{(ki)!(n+i)!}\int_{0}^{\infty}y^{n+i(k+1)} e^{-(\lambda+k\mu+c_{1}z^{\alpha_{1}}+c_{2}z^{\alpha_{2}})y}\mathrm{d}y\nonumber \\ 
	&\ \  +\sum_{i=0}^{\infty}\frac{\lambda^{n+i}(k\mu)^{ki+1}}{(ki)!(n+i)!}\int_{0}^{\infty}\int_{0}^{y}p_{0}(u)(y-u)^{n+i(k+1)}e^{-(\lambda+k\mu)(y-u)}\nonumber\\
	&\ \ \cdot e^{-(c_{1}z^{\alpha_{1}}+c_{2}z^{\alpha_{2}})y}\mathrm{d}u\mathrm{d}y -\sum_{i=0}^{\infty}\frac{\lambda^{n+i+1}(k\mu)^{k(i+1)}}{(k(i+1)-1)!(n+i+1)!}\nonumber\\
	&\ \ \cdot \int_{0}^{\infty}\int_{0}^{y}p_{0}(u)(y-u)^{n+k+i(k+1)}e^{-(\lambda+k\mu)(y-u)} e^{-(c_{1}z^{\alpha_{1}}+c_{2}z^{\alpha_{2}})y}\mathrm{d}u\mathrm{d}y\Big).
\end{align}
By using Lemma 5.4 of Ascione {\it et al.} (2020) and \eqref{constnts12} in \eqref{pnslp121} and \eqref{pnklp121}, we get the required result.	
\end{proof}
\begin{remark}
	For $c_{1}=1$ or $c_{2}=1$ in \eqref{pnslp12}, the Laplace transform of state probabilities of mixed time-changed Erlang queue reduces to that of fractional Erlang queue (see Ascione {\it et al.} (2020), Theorem 5.5).
\end{remark}
\begin{theorem}\label{thmpnst12}
	For $1\leq s\leq k$, the state probabilities of $\{\mathcal{Q}^{\alpha_{1},\alpha_{2}}(t)\}_{t\geq0}$ are given by
\begin{align*}
	p_{n,s}^{\alpha_{1},\alpha_{2}}(t)&=\sum_{i=0}^{\infty}\frac{B^{n,s}_{i}}{c_{1}^{\delta_{i}^{n,s}}}\Big(c_{1}f_{\delta_{i}^{n,s}}^{*(\delta_{i}^{n,s})}(t)+c_{2}g_{\delta_{i}^{n,s}}^{*(b^{n,s}_{i})}(t)\Big)\nonumber\\
	&\ \ +\sum_{i=0}^{\infty}\sum_{m=1}^{\infty}\sum_{r=0}^{\infty}\frac{C^{n,s}_{i,m,r}}{c_{1}^{\nu_{i,m,r}^{n,s}}}\Big(c_{1}f_{\nu_{i,m,r}^{n,s}}^{*(\nu_{i,m,r}^{n,s})}(t)+c_{2}g_{\nu_{i,m,r}^{n,s}}^{*(\nu_{i,m,r}^{n,s})}(t)\Big)\nonumber\\
	&\ \ -\sum_{i=0}^{\infty}\sum_{m=1}^{\infty}\sum_{r=0}^{\infty}\frac{D^{n,s}_{i,m,r}}{c_{1}^{\pi_{i,m,r}^{n,s}}}\Big(c_{1}f_{\pi_{i,m,r}^{n,s}}^{*(\pi_{i,m,r}^{n,s})}(t)+c_{2}g_{\pi_{i,m,r}^{n,s}}^{*(\pi_{i,m,r}^{n,s})}(t)\Big),\,n\geq1,
\end{align*}
where $f_{N}(\cdot)$ and $g_{N}(\cdot)$ are given in \eqref{f12} and \eqref{g12}, respectively. 
\end{theorem}
\begin{proof}
On taking the inverse Laplace transform of \eqref{pnslp12}, we obtain
\begin{align}\label{pns1212}
	p_{n,s}^{\alpha_{1},\alpha_{2}}(t)&=\sum_{i=0}^{\infty}B^{n,s}_{i}\Bigg(\frac{1}{c_{1}^{\delta_{i}^{n,s}-1}}\mathbb{L}^{-1}\Bigg(\bigg(\frac{z^{\frac{\alpha_{1}-1}{\delta_{i}^{n,s}}}}{z^{\alpha_{1}}+\frac{c_{2}}{c_{1}}z^{\alpha_{2}}+\frac{\lambda+k\mu}{c_{1}}}\bigg)^{\delta_{i}^{n,s}};t\Bigg)\nonumber \\
	&\ \ +\frac{c_{2}}{c_{1}^{\delta_{i}^{n,s}}}\mathbb{L}^{-1}\Bigg(\bigg(\frac{z^{\frac{\alpha_{2}-1}{\delta_{i}^{n,s}}}}{z^{\alpha_{1}}+\frac{c_{2}}{c_{1}}z^{\alpha_{2}}+\frac{\lambda+k\mu}{c_{1}}}\bigg)^{\delta_{i}^{n,s}};t\Bigg)\Bigg)\nonumber\\
	&\ \ +\sum_{i=0}^{\infty}\sum_{m=1}^{\infty}\sum_{r=0}^{\infty}C^{n,s}_{i,m,r}\Bigg(\frac{1}{c_{1}^{\nu_{i,m,r}^{n,s}-1}}\mathbb{L}^{-1}\Bigg(\bigg(\frac{z^{\frac{\alpha_{1}-1}{\nu_{i,m,r}^{n,s}}}}{z^{\alpha_{1}}+\frac{c_{2}}{c_{1}}z^{\alpha_{2}}+\frac{\lambda+k\mu}{c_{1}}}\bigg)^{\nu_{i,m,r}^{n,s}};t\Bigg)\nonumber \\
	&\ \ +\frac{c_{2}}{c_{1}^{\nu_{i,m,r}^{n,s}}}\mathbb{L}^{-1}\Bigg(\bigg(\frac{z^{\frac{\alpha_{2}-1}{\nu_{i,m,r}^{n,s}}}}{z^{\alpha_{1}}+\frac{c_{2}}{c_{1}}z^{\alpha_{2}}+\frac{\lambda+k\mu}{c_{1}}}\bigg)^{\nu_{i,m,r}^{n,s}};t\Bigg)\Bigg)\nonumber\\
	&\ \ -\sum_{i=0}^{\infty}\sum_{m=1}^{\infty}\sum_{r=0}^{\infty}D^{n,s}_{i,m,r}\Bigg(\frac{1}{c_{1}^{\pi_{i,m,r}^{n,s}-1}}\mathbb{L}^{-1}\Bigg(\bigg(\frac{z^{\frac{\alpha_{1}-1}{\pi_{i,m,r}^{n,s}}}}{z^{\alpha_{1}}+\frac{c_{2}}{c_{1}}z^{\alpha_{2}}+\frac{\lambda+k\mu}{c_{1}}}\bigg)^{\pi_{i,m,r}^{n,s}};t\Bigg)\nonumber \\
	&\ \ +\frac{c_{2}}{c_{1}^{\pi_{i,m,r}^{n,s}}}\mathbb{L}^{-1}\Bigg(\bigg(\frac{z^{\frac{\alpha_{2}-1}{\pi_{i,m,r}^{n,s}}}}{z^{\alpha_{1}}+\frac{c_{2}}{c_{1}}z^{\alpha_{2}}+\frac{\lambda+k\mu}{c_{1}}}\bigg)^{\pi_{i,m,r}^{n,s}};t\Bigg)\Bigg).
\end{align}
On using \eqref{rholp} in \eqref{pns1212}, we get the required result.
\end{proof}
\begin{remark}
	For $c_{1}=1$ in Theorem \ref{thmpnst12}, the state probabilities of $\{\mathcal{Q}^{\alpha_{1},\alpha_{2}}(t)\}_{t\geq0}$ reduces to that of fractional Erlang queue (see Ascione {\it et al.} (2020), Theorem 5.5).
\end{remark}
\subsection{Queue length process}
Here, we define the queue length process of mixed time-changed Erlang queue as follows:  
\begin{equation*}
	\mathcal{L}^{\alpha_{1},\alpha_{2}}(t)\coloneqq g_{k}(\mathcal{Q}^{\alpha_{1},\alpha_{2}}(t)),\, t\geq0,
\end{equation*}
where $g_{k}(\cdot)$ is a bijective map as defined in \eqref{gk12}. 
Let $p_{n}^{\alpha_{1},\alpha_{2}}(t)=\mathrm{Pr}(\mathcal{L}^{\alpha_{1},\alpha_{2}}(t)=n|\mathcal{L}^{\alpha_{1},\alpha_{2}}(0)=0)$, $n\geq0$
be its state probabilities such that $p_{-n}^{\alpha_{1},\alpha_{2}}(t)=0$. 

The mean queue length of mixed time-changed Erlang queue is defined as
\begin{equation}\label{ML12}
	\mathcal{M}^{\alpha_{1},\alpha_{2}}(t)\coloneqq\mathbb{E}(\mathcal{L}^{\alpha_{1},\alpha_{2}}(t)|\mathcal{L}^{\alpha_{1},\alpha_{2}}(0)=0).
\end{equation}

As $g_{k}(\cdot)$ is a bijective map, the mixed time-changed Erlang queue $\{\mathcal{Q}^{\alpha_{1},\alpha_{2}}(t)\}_{t\geq0}$ is equivalently represented by the queue length process $\{\mathcal{L}^{\alpha_{1},\alpha_{2}}(t)\}_{t\geq0}$. Thus, 
\begin{equation}\label{pngk12}
	p^{\alpha_{1},\alpha_{2}}_{n,s}(t)=p^{\alpha_{1},\alpha_{2}}_{g_{k}(n,s)}(t)
\end{equation}
and
\begin{equation}\label{pab212}
	p_{n}^{\alpha_{1},\alpha_{2}}(t)=p^{\alpha_{1},\alpha_{2}}_{a_{k}(n),b_{k}(n)}(t),
\end{equation}
where $a_{k}(\cdot)$ and $b_{k}(\cdot)$ are as given in \eqref{ak12} and \eqref{bk12}, respectively.

Now, by using \eqref{pngk12} in Theorem \ref{DE12}, we get the following system of differential equations that governs the state probabilities of $\{\mathcal{L}^{\alpha_{1},\alpha_{2}}(t)\}_{t\geq0}$:
{\small\begin{align}\label{DEL12}
	\left.
	\begin{aligned}
		\Big(c_{1}\frac{\mathrm{d}^{\alpha_{1}}}{\mathrm{d}t^{\alpha_{1}}}+c_{2}\frac{\mathrm{d}^{\alpha_{2}}}{\mathrm{d}t^{\alpha_{2}}}\Big)p_{0}^{\alpha_{1},\alpha_{2}}(t)&=-\lambda p_{0}^{\alpha_{1},\alpha_{2}}(t)+k\mu p_{1}^{\alpha_{1},\alpha_{2}}(t),\\
		\Big(c_{1}\frac{\mathrm{d}^{\alpha_{1}}}{\mathrm{d}t^{\alpha_{1}}}+c_{2}\frac{\mathrm{d}^{\alpha_{2}}}{\mathrm{d}t^{\alpha_{2}}}\Big)p_{n}^{\alpha_{1},\alpha_{2}}(t)&=-(\lambda+k\mu) p_{n}^{\alpha_{1},\alpha_{2}}(t)+k\mu p_{n+1}^{\alpha_{1},\alpha_{2}}(t),\,1\leq n\leq k-1,\\
		\Big(c_{1}\frac{\mathrm{d}^{\alpha_{1}}}{\mathrm{d}t^{\alpha_{1}}}+c_{2}\frac{\mathrm{d}^{\alpha_{2}}}{\mathrm{d}t^{\alpha_{2}}}\Big)p_{n}^{\alpha_{1},\alpha_{2}}(t)&=-(\lambda+k\mu) p_{n}^{\alpha_{1},\alpha_{2}}(t)+k\mu p_{n+1}^{\alpha_{1},\alpha_{2}}(t)+\lambda p_{n-k}^{\alpha_{1},\alpha_{2}}(t),\,n\geq k
	\end{aligned}
	\right\}
\end{align}}
with initial condition $p^{\alpha_{1},\alpha_{2}}_{0}(0)=1$. Here, $\frac{\mathrm{d}^{\alpha_{i}}}{\mathrm{d}t^{\alpha_{i}}}$ is a Caputo fractional derivative of order $\alpha_{i}$ defined in \eqref{caputo}.
	
\begin{theorem}\label{thmmalpha12}
	The mean queue length of mixed time-changed Erlang queue solves the following fractional differential equation:
\begin{equation}\label{meanDE12}
	\Big(c_{1}\frac{\mathrm{d}^{\alpha_{1}}}{\mathrm{d}t^{\alpha_{1}}}+c_{2}\frac{\mathrm{d}^{\alpha_{2}}}{\mathrm{d}t^{\alpha_{2}}}\Big)\mathcal{M}^{\alpha_{1},\alpha_{2}}(t)=k(\lambda-\mu)+k\mu p_{0}^{\alpha_{1},\alpha_{2}}(t),
\end{equation}
with initial condition $\mathcal{M}^{\alpha_{1},\alpha_{2}}(0)=0$.
\end{theorem}
\begin{proof}
From \eqref{ML12}, we have
\begin{equation}\label{dtmt12}
	\mathcal{M}^{\alpha_{1},\alpha_{2}}(t)=\sum_{n=0}^{\infty}np_{n}^{\alpha_{1},\alpha_{2}}(t).
\end{equation}
Note that $\mathcal{M}^{\alpha_{1},\alpha_{2}}(0)=0$ as $p_{0}^{\alpha_{1},\alpha_{2}}(0)=1$. 
On using \eqref{DEL12}, we get
\begin{align}
	\Big(c_{1}\frac{\mathrm{d}^{\alpha_{1}}}{\mathrm{d}t^{\alpha_{1}}}+c_{2}\frac{\mathrm{d}^{\alpha_{2}}}{\mathrm{d}t^{\alpha_{2}}}\Big)\mathcal{M}^{\alpha_{1},\alpha_{2}}(t)
	&=-(\lambda + k \mu)\mathcal{M}^{\alpha_{1},\alpha_{2}}(t)+k\mu\sum_{n=1}^{\infty}np_{n+1}^{\alpha_{1},\alpha_{2}}(t)+\lambda\sum_{n=k}^{\infty}np_{n-k}^{\alpha_{1},\alpha_{2}}(t)\nonumber\\
	&=-(\lambda + k \mu)\mathcal{M}^{\alpha_{1},\alpha_{2}}(t)+k\mu(\mathcal{M}^{\alpha_{1},\alpha_{2}}(t)+p_{0}^{\alpha_{1},\alpha_{2}}(t)-1)\nonumber\\
	&\ \ +\lambda(\mathcal{M}^{\alpha_{1},\alpha_{2}}(t)+k)\nonumber\\
	&=k(\lambda-\mu)+k\mu p_{0}^{\alpha_{1},\alpha_{2}}(t).
\end{align}
This proves the result.	
\end{proof}
\begin{remark}
	For $c_{1}=1$ or $c_{2}=1$ in \eqref{meanDE12}, it reduces to the governing differential equation for mean queue length of fractional Erlang queue (see Ascione {\it et al.} (2020), Eq. (42)).
\end{remark}
\begin{theorem}
Let $E_{\alpha,\beta}^{\gamma}(\cdot)$ be three parameter Mittag-Leffler function as defined in \eqref{Mittag12} and
\begin{equation}\label{hN12}
	h_{N}(t)=t^{\alpha_{1}-\big(\frac{N-1}{N}\big)}\sum_{j=0}^{\infty}\Big(-\frac{c_{2}}{c_{1}}\Big)^{j}t^{(\alpha_{1}-\alpha_{2})j}E^{j+1}_{\alpha_{1},\alpha_{1}+(\alpha_{1}-\alpha_{2})j+\frac{1}{N}}\Big(-\Big(\frac{\lambda+k\mu}{c_{1}}\Big)t^{\alpha_{1}}\Big).
\end{equation}
Then, the mean queue length of mixed time-changed Erlang queue is given by  
\begin{align}\label{malpahat12}
	\mathcal{M}^{\alpha_{1},\alpha_{2}}(t)&=\frac{k(\lambda-\mu)}{c_{1}}t^{\alpha_{1}}E^{1}_{\alpha_{1}-\alpha_{2},\alpha_{1}+1}\Big(-\frac{c_{2}}{c_{1}}t^{\alpha_{1}-\alpha_{2}}\Big)+k\mu\sum_{m=1}^{\infty}\sum_{r=0}^{\infty}\frac{A_{m,r}^{0}}{c_{1}^{a_{m,r}^{0}}}h_{a_{m,r}^{0}}^{*(a_{m,r}^{0})}(t),
\end{align}
where $a_{m,r}^{0}$ and $A_{m,r}^{0}$ are given in \eqref{amr12} and \eqref{Amr12}, respectively.
\end{theorem}
\begin{proof}
By using \eqref{pnsalphat12} and \eqref{pab212} in \eqref{dtmt12}, we get
\begin{align}\label{ML121}
	\mathcal{M}^{\alpha_{1},\alpha_{2}}(t)&=\int_{0}^{\infty}\mathcal{M}(y)\mathrm{Pr}\{Y_{\alpha_{1},\alpha_{2}}(t)\in\mathrm{d}y\}\nonumber\\
	&=k(\lambda-\mu)\int_{0}^{\infty}y\mathrm{Pr}\{Y_{\alpha_{1},\alpha_{2}}(t)\in\mathrm{d}y\}+k\mu\int_{0}^{\infty}\int_{0}^{y}p_{0}(u)\mathrm{d}u\mathrm{Pr}\{Y_{\alpha_{1},\alpha_{2}}(t)\in\mathrm{d}y\},
\end{align}
where the second equality follows from Eq. (21) of Luchak (1958).
On taking the Laplace transform on both sides of \eqref{ML121} and using \eqref{mixedlp12}, we get
\begin{align}\label{ML1212}
	\tilde{\mathcal{M}}^{\alpha_{1},\alpha_{2}}(z)
    &=(c_{1}z^{\alpha_{1}-1}+c_{2}z^{\alpha_{2}-1})\Big(k(\lambda-\mu)\int_{0}^{\infty}ye^{-(c_{1}z^{\alpha_{1}}+c_{2}z^{\alpha_{2}})y}\mathrm{d}y\nonumber\\ 
	&\hspace{7.2cm} +k\mu\int_{0}^{\infty}\int_{0}^{y}p_{0}(u)e^{-(c_{1}z^{\alpha_{1}}+c_{2}z^{\alpha_{2}})y}\mathrm{d}u\mathrm{d}y\Big)\nonumber\\
	&=(c_{1}z^{\alpha_{1}-1}+c_{2}z^{\alpha_{2}-1})\Big(\frac{k(\lambda-\mu)}{(c_{1}z^{\alpha_{1}}+c_{2}z^{\alpha_{2}})^{2}}+k\mu\int_{0}^{\infty}p_{0}(u)\int_{u}^{\infty}e^{-(c_{1}z^{\alpha_{1}}+c_{2}z^{\alpha_{2}})y}\mathrm{d}y\mathrm{d}u\Big)\nonumber\\
	&=(c_{1}z^{\alpha_{1}-1}+c_{2}z^{\alpha_{2}-1})\Big(\frac{k(\lambda-\mu)}{(c_{1}z^{\alpha_{1}}+c_{2}z^{\alpha_{2}})^{2}}\nonumber\\
	&\hspace{6.1cm} +\frac{k\mu}{(c_{1}z^{\alpha_{1}}+c_{2}z^{\alpha_{2}})}\int_{0}^{\infty}p_{0}(u) e^{-(c_{1}z^{\alpha_{1}}+c_{2}z^{\alpha_{2}})u}\mathrm{d}u\Big)\nonumber\\
	&=\frac{k(\lambda-\mu)}{c_{1}z^{\alpha_{1}+1}+c_{2}z^{\alpha_{2}+1}}+\frac{k\mu}{z}\int_{0}^{\infty}p_{0}(u) e^{-(c_{1}z^{\alpha_{1}}+c_{2}z^{\alpha_{2}})y}\mathrm{d}u\nonumber\\
	&=\frac{k(\lambda-\mu)}{c_{1}z^{\alpha_{1}+1}+c_{2}z^{\alpha_{2}+1}}+\frac{k\mu}{z}\frac{\tilde{p_{0}}^{\alpha_{1},\alpha_{2}}(z)}{(c_{1}z^{\alpha_{1}-1}+c_{2}z^{\alpha_{2}-1})}\,\,(\text{by using \eqref{p0alphalp12}})\nonumber\\
	&=\frac{k(\lambda-\mu)}{c_{1}z^{\alpha_{1}+1}+c_{2}z^{\alpha_{2}+1}}+\frac{k\mu}{z}\sum_{m=1}^{\infty}\sum_{r=0}^{\infty}\frac{A_{m,r}^{0}}{(\lambda+k\mu+c_{1}z^{\alpha_{1}}+c_{2}z^{\alpha_{2}})^{a_{m,r}^{0}}},
\end{align}
where we have used \eqref{p0lp121} in the last step.

On taking the inverse Laplace transform on both sides of \eqref{ML1212}, we get
\begin{align*}
	\mathcal{M}^{\alpha_{1},\alpha_{2}}(t)&=\frac{k(\lambda-\mu)}{c_{1}}\mathbb{L}^{-1}\Big(\frac{z^{-\alpha_{2}-1}}{z^{\alpha_{1}-\alpha_{2}}+\frac{c_{2}}{c_{1}}};t\Big)\nonumber\\
	&\ \  +k\mu\sum_{m=1}^{\infty}\sum_{r=0}^{\infty} \frac{A_{m,r}^{0}}{c_{1}^{a_{m,r}^{0}}}\mathbb{L}^{-1}\bigg(\bigg(\frac{z^{\frac{-1}{a_{m,r}^{0}}}}{z^{\alpha_{1}}+\frac{c_{2}}{c_{1}}z^{\alpha_{2}}+\frac{\lambda+k\mu}{c_{1}}}\bigg)^{a_{m,r}^{0}};t\bigg).
\end{align*}
From \eqref{ltm12} and \eqref{rholp}, the required result follows.
\end{proof}
\begin{remark}
	For $c_{1}=1$ in \eqref{malpahat12}, we get the mean queue length of fractional Erlang queue (see Ascione {\it et al.} (2020), Eq. (47)).
\end{remark}
\subsection{Inter-arrival and inter-phase times of mixed time changed Erlang queue}
Here, we view the mixed time-changed Erlang queue in terms of its inter-arrival and service times. 

Let us denote the following:
\begin{align*}
	q_{n}^{\alpha_{1},\alpha_{2}}(t)&=\mathrm{Pr}(\mathcal{N}^{\alpha_{1},\alpha_{2}}(t)=n|\mathcal{Q}^{\alpha_{1},\alpha_{2}}(0)=(0,0)),\,n\geq1,\\
	r_{s}^{\alpha_{1},\alpha_{2}}(t)&=\mathrm{Pr}(\mathcal{S}^{\alpha_{1},\alpha_{2}}(t)=s|\mathcal{Q}^{\alpha_{1},\alpha_{2}}(0)=(0,0)),\,1\leq s\leq k.
\end{align*}
Then, 
\begin{equation*}
	q_{n}^{\alpha_{1},\alpha_{2}}(t)=\sum_{s=1}^{k}p_{n,s}^{\alpha_{1},\alpha_{2}}(t) \text{ and }
	r_{s}^{\alpha_{1},\alpha_{2}}(t)=\sum_{n=1}^{\infty}p_{n,s}^{\alpha_{1},\alpha_{2}}(t).
\end{equation*}
For $1\leq s\leq k$, we add the differential equations given in Theorem \ref{DE12} to obtain
\begin{align*}
	\Big(c_{1}\frac{\mathrm{d}^{\alpha_{1}}}{\mathrm{d}t^{\alpha_{1}}}+c_{2}\frac{\mathrm{d}^{\alpha_{2}}}{\mathrm{d}t^{\alpha_{2}}}\Big)p_{0}^{\alpha_{1},\alpha_{2}}(t)&=-\lambda p_{0}^{\alpha_{1},\alpha_{2}}(t)+k\mu p_{1,1}^{\alpha_{1},\alpha_{2}}(t),\\
	\Big(c_{1}\frac{\mathrm{d}^{\alpha_{1}}}{\mathrm{d}t^{\alpha_{1}}}+c_{2}\frac{\mathrm{d}^{\alpha_{2}}}{\mathrm{d}t^{\alpha_{2}}}\Big)q_{1}^{\alpha_{1},\alpha_{2}}(t)&=-\lambda q_{1}^{\alpha_{1},\alpha_{2}}(t)+k\mu(p_{2,1}^{\alpha_{1},\alpha_{2}}(t)-p_{1,1}^{\alpha_{1},\alpha_{2}}(t))+\lambda p_{0}^{\alpha_{1},\alpha_{2}}(t),\\
	\Big(c_{1}\frac{\mathrm{d}^{\alpha_{1}}}{\mathrm{d}t^{\alpha_{1}}}+c_{2}\frac{\mathrm{d}^{\alpha_{2}}}{\mathrm{d}t^{\alpha_{2}}}\Big)q_{n}^{\alpha_{1},\alpha_{2}}(t)&=-\lambda q_{n}^{\alpha_{1},\alpha_{2}}(t)+k\mu(p_{n+1,1}^{\alpha_{1},\alpha_{2}}(t)-p_{n,1}^{\alpha_{1},\alpha_{2}}(t))+\lambda q_{n-1}^{\alpha_{1},\alpha_{2}}(t),\,n\geq 2,
\end{align*} 
with $p_{0}^{\alpha_{1},\alpha_{2}}(0)=1$ and $q_{n}^{\alpha_{1},\alpha_{2}}(0)=0$, $n\geq1$.

Similarly, for the Erlang queue, we have (see Ascione {\it et al.} (2020), p. 3257)
\begin{align*}
    \frac{\mathrm{d}}{\mathrm{d}t}p_{0}(t)&=-\lambda p_{0}(t)+k\mu p_{1,1}(t),\\
	\frac{\mathrm{d}}{\mathrm{d}t}q_{1}(t)&=-\lambda q_{1}(t)+k\mu(p_{2,1}(t)-p_{1,1}(t))+\lambda p_{0}(t),\\
	\frac{\mathrm{d}}{\mathrm{d}t}q_{n}(t)&=-\lambda q_{n}(t)+k\mu(p_{n+1,1}(t)-p_{n,1}(t))+\lambda q_{n-1}(t),\,n\geq 2,
\end{align*}
with $p_{0}(0)=1$ and $q_{n}(0)=0$, $n\geq1$. 
Here,  
\begin{align*}
	q_{n}(t)&=\mathrm{Pr}(\mathcal{N}(t)=n|\mathcal{Q}(0)=(0,0)),\,n\geq1,\\
	r_{s}(t)&=\mathrm{Pr}(\mathcal{S}(t)=s|\mathcal{Q}(0)=(0,0)),\,1\leq s\leq k.
\end{align*}
Note that $q_{n}(t)=\sum_{s=1}^{k}p_{n,s}(t) \text{ and }
r_{s}(t)=\sum_{n=1}^{\infty}p_{n,s}(t)$.

Recall that the arrival time is the random time at which the first arrival occurs, and the inter-arrival time is defined as the time between two successive arrivals.
\begin{theorem}\label{thminterarrival}
The inter-arrival times $T^{\alpha_{1},\alpha_{2}}$ of the mixed time-changed Erlang queue $\{\mathcal{Q}^{\alpha_{1},\alpha_{2}}(t)\}_{t\geq0}$ are independent and identically distributed (iid) with the following distribution:
\begin{align}\label{intrarr12}
	\mathrm{Pr}(T^{\alpha_{1},\alpha_{2}}>t)	&=\sum_{r=0}^{\infty}\Big(-\frac{c_{2}}{c_{1}}\Big)^{r}t^{(\alpha_{1}-\alpha_{2})r}E^{r+1}_{\alpha_{1},(\alpha_{1}-\alpha_{2})r+1}\Big(-\frac{\lambda}{c_{1}}t^{\alpha_{1}}\Big)\nonumber\\
	&\hspace{1.8cm} -\sum_{r=0}^{\infty}\Big(-\frac{c_{2}}{c_{1}}\Big)^{r+1}t^{(\alpha_{1}-\alpha_{2})(r+1)}E^{r+1}_{\alpha_{1},(\alpha_{1}-\alpha_{2})(r+1)+1}\Big(-\frac{\lambda}{c_{1}}t^{\alpha_{1}}\Big),
	\end{align}
where $E_{\alpha,\beta}^{\gamma}(\cdot)$ is three parameter Mittag-Leffler function as defined in \eqref{Mittag12}.
\end{theorem}
\begin{proof}
Note that the inter-arrival times of mixed time-changed Erlang queue are independent. To obtain the distribution of inter-arrival times, let us define a process $\{\mathcal{N}_{*}(t)\}_{t\geq0}$ whose state probabilities ${\textbf q}_{n}(t)=\mathrm{Pr}(\mathcal{N}_{*}(t)=n)$, $n\geq0$ solve the following Cauchy problem:
\begin{align}\label{qmDE12}
	\left.
	\begin{aligned}
		\frac{\mathrm{d}}{\mathrm{d}t}{\textbf q}_{m}(t)&=-\lambda{\textbf q}_{m}(t),\\
		\frac{\mathrm{d}}{\mathrm{d}t}{\textbf q}_{m+1}(t)&=\lambda{\textbf q}_{m}(t),\\
		\frac{\mathrm{d}}{\mathrm{d}t}{\textbf q}_{n}(t)&=0,\,n\geq0,\,n\neq m, \,m+1,
	\end{aligned}
	\right\}
\end{align}
with initial conditions ${\textbf q}_{m}(0)=1$ and ${\textbf q}_{n}(0)=0$ for $n\geq0$, $n\neq m$. Here, $\{\mathcal{N}_{*}(t)\}_{t\geq0}$ is a counting process with the same arrival rate $\lambda$ as that of $\{\mathcal{N}(t)\}_{t\geq0}$. It starts at $m$ with $m+1$ as its absorbing state. Observe that ${\textbf q}_{m}(t)+{\textbf q}_{m+1}(t)=1$, $t\geq0$. 

Let us denote $T$ as the arrival time of first new customer in the Erlang queue $\{\mathcal{Q}(t)\}_{t\geq0}$ starting from $\mathcal{N}_{*}(0)=m$. Then, its distribution function $F_{T}(\cdot)$ is given by
\begin{align*}
	F_{T}(t)\coloneq\mathrm{Pr}(T\leq t)&=1-\mathrm{Pr}(T>t)\\
	&=1-\mathrm{Pr}(\mathcal{N}_{*}(t)=m)\\
	&=\mathrm{Pr}(\mathcal{N}_{*}(t)=m+1)=\textbf{q}_{m+1}(t).
\end{align*} 

Let $\mathcal{N}_{*}^{\alpha_{1},\alpha_{2}}(t)=\mathcal{N}_{*}(Y_{\alpha_{1},\alpha_{2}}(t))$, $t\geq0$ be the time-changed process whose state probabilities are denoted by 
${\textbf q}_{m}^{\alpha_{1},\alpha_{2}}(t)=\mathrm{Pr}(\mathcal{N}_{*}^{\alpha_{1},\alpha_{2}}(t)=m) \text{ and } {\textbf q}_{m+1}^{\alpha_{1},\alpha_{2}}(t)=\mathrm{Pr}(\mathcal{N}_{*}^{\alpha_{1},\alpha_{2}}(t)=m+1)$.
So, $\{\mathcal{N}_{*}^{\alpha_{1},\alpha_{2}}(t)\}_{t\geq0}$ is a process which starts at $m$ with $m+1$ being the absorbing state. Let $T^{\alpha_{1},\alpha_{2}}$ be the arrival time of the first new customer in the mixed time-changed Erlang queue $\{\mathcal{Q}^{\alpha_{1},\alpha_{2}}(t)\}_{t\geq0}$ starting from $\mathcal{N}_{*}^{\alpha_{1},\alpha_{2}}(0)=m$. Thus, $F_{T^{\alpha_{1},\alpha_{2}}}(t)=\mathrm{Pr}(T^{\alpha_{1},\alpha_{2}}\leq t)={\textbf q}_{m+1}^{\alpha_{1},\alpha_{2}}(t)$. 

As $\{\mathcal{N}(t)\}_{t\geq0}$ is a Markov process, $\{\mathcal{N}^{\alpha_{1},\alpha_{2}}(t)\}_{t\geq0}$ is a semi-Markov process. Let $T_{n}$ be its $n$th jump time. Then, discrete time process $\mathcal{N}^{\alpha_{1},\alpha_{2}}(T_{n})$ is a time-homogeneous Markov chain (see Gikhman and Skorohod (1975)). Thus, the inter-arrival times of mixed time-changed Erlang queue $\{\mathcal{Q}^{\alpha_{1},\alpha_{2}}(t)\}_{t\geq0}$ are distributed according to $T^{\alpha_{1},\alpha_{2}}$.

Now, we show that the state probabilities ${\textbf q}_{m}^{\alpha_{1},\alpha_{2}}(t)$ and ${\textbf q}_{m+1}^{\alpha_{1},\alpha_{2}}(t)$ of $\{\mathcal{N}_{*}^{\alpha_{1},\alpha_{2}}(t)\}_{t\geq0}$ solves the following differential equations:
\begin{align}
	\Big(c_{1}\frac{\mathrm{d}^{\alpha_{1}}}{\mathrm{d}t^{\alpha_{1}}}+c_{2}\frac{\mathrm{d}^{\alpha_{2}}}{\mathrm{d}t^{\alpha_{2}}}\Big){\textbf q}_{m}^{\alpha_{1},\alpha_{2}}(t)&=-\lambda{\textbf q}_{m}^{\alpha_{1},\alpha_{2}}(t),\label{tqm12}\\
	\Big(c_{1}\frac{\mathrm{d}^{\alpha_{1}}}{\mathrm{d}t^{\alpha_{1}}}+c_{2}\frac{\mathrm{d}^{\alpha_{2}}}{\mathrm{d}t^{\alpha_{2}}}\Big){\textbf q}_{m+1}^{\alpha_{1},\alpha_{2}}(t)&=\lambda{\textbf q}_{m}^{\alpha_{1},\alpha_{2}}(t)\label{tqmm12},
\end{align}
with ${\textbf q}_{m}^{\alpha_{1},\alpha_{2}}(0)=1$ and ${\textbf q}_{m+1}^{\alpha_{1},\alpha_{2}}(0)=0$.

The Eq. \eqref{tqm12} holds true if and only if the Laplace transform $\tilde{\textbf{q}}_{m}^{\alpha_{1},\alpha_{2}}(z)$ of ${\textbf q}_{m}^{\alpha_{1},\alpha_{2}}(t)$ solves:
\begin{equation}\label{qmDElp12}
	(c_{1}z^{\alpha_{1}-1}+c_{2}z^{\alpha_{2}-1})(z\tilde{{\textbf q}}_{m}^{\alpha_{1},\alpha_{2}}(z)-1)=-\lambda\tilde{{\textbf q}}_{m}^{\alpha_{1},\alpha_{2}}(z),\,z>0,
\end{equation}
where we have used \eqref{caputolp}.

On taking the Laplace transform on both sides of
\begin{equation*}\label{qm12}
	{\textbf q}_{m}^{\alpha_{1},\alpha_{2}}(t)=\mathrm{Pr}(\mathcal{N}_{*}(Y_{\alpha_{1},\alpha_{2}}(t))=m)
	=\int_{0}^{\infty}{\textbf q}_{m}(y)\mathrm{Pr}\{Y_{\alpha_{1},\alpha_{2}}(t)\in \mathrm{d}y\}
\end{equation*}
and by using \eqref{mixedlp12}, we get 
\begin{equation}\label{lpqm12}
	\tilde{{\textbf q}}_{m}^{\alpha_{1},\alpha_{2}}(z)=(c_{1}z^{\alpha_{1}-1}+c_{2}z^{\alpha_{2}-1})\int_{0}^{\infty}{\textbf q}_{m}(y)e^{-y(c_{1}z^{\alpha_{1}}+c_{2}z^{\alpha_{2}})} \mathrm{d}y.
\end{equation}
On substituting \eqref{lpqm12} in \eqref{qmDElp12}, we obtain
\begin{equation}\label{qmDElp121}
    (c_{1}z^{\alpha_{1}}+c_{2}z^{\alpha_{2}})\int_{0}^{\infty}{\textbf q}_{m}(y)e^{-y(c_{1}z^{\alpha_{1}}+c_{2}z^{\alpha_{2}})} \mathrm{d}y-1=-\lambda\int_{0}^{\infty}{\textbf q}_{m}(y)e^{-y(c_{1}z^{\alpha_{1}}+c_{2}z^{\alpha_{2}})} \mathrm{d}y.
\end{equation}
By using \eqref{qmDE12} in \eqref{qmDElp121} and on solving it, we get an identity. This establishes that \eqref{tqm12} holds true. Similarly, it can be shown that \eqref{tqmm12} holds true. 

On using Eq. (2.41) of Beghin (2012), the Eq. \eqref{tqm12} along with the initial condition ${\textbf q}_{m}^{\alpha_{1},\alpha_{2}}(0)=1$ can be solved to get the required result. This completes the proof.
\end{proof}
\begin{remark}
From Beghin (2012), it follows that the Laplace transform of ${\textbf q}_{m}^{\alpha_{1},\alpha_{2}}(t)=\mathrm{Pr}(T^{\alpha_{1},\alpha_{2}}>t)$ is 
\begin{equation*}
	\tilde{{\textbf q}}_{m}^{\alpha_{1},\alpha_{2}}(z)=\frac{c_{1}z^{\alpha_{1}-1}+c_{2}z^{\alpha_{2}-1}}{\lambda+c_{1}z^{\alpha_{1}}+c_{2}z^{\alpha_{2}}}.
\end{equation*}
\end{remark}
\begin{remark}
	For $c_{1}=1$ in Theorem \ref{thminterarrival}, the distribution of inter-arrival times of mixed time-changed Erlang queue reduces to that of fractional Erlang queue (see Ascione {\it et al.} (2020), Theorem 4.1).
\end{remark}
The proof of the following result follows similar lines to that of Theorem \ref{thminterarrival}. Thus, it is omitted. 
\begin{theorem}\label{thminterphase12}
The inter-phase times $\mathcal{P}^{\alpha_{1},\alpha_{2}}$ of mixed time-changed Erlang queue are iid with the following distribution:
\begin{align}\label{intephase12}
	\mathrm{Pr}(\mathcal{P}^{\alpha_{1},\alpha_{2}}>t)&=\sum_{r=0}^{\infty}\Big(-\frac{c_{2}}{c_{1}}\Big)^{r}t^{(\alpha_{1}-\alpha_{2})r}E^{r+1}_{\alpha_{1},(\alpha_{1}-\alpha_{2})r+1}\Big(-\frac{k\mu}{c_{1}}t^{\alpha_{1}}\Big)\nonumber\\
	&\hspace{1.7cm} -\sum_{r=0}^{\infty}\Big(-\frac{c_{2}}{c_{1}}\Big)^{r+1}t^{(\alpha_{1}-\alpha_{2})(r+1)}E^{r+1}_{\alpha_{1},(\alpha_{1}-\alpha_{2})(r+1)+1}\Big(-\frac{k\mu}{c_{1}}t^{\alpha_{1}}\Big),
\end{align}
where $E_{\alpha,\beta}^{\gamma}(\cdot)$ is three parameter Mittag-Leffler function as defined in \eqref{Mittag12}.
\end{theorem}
\begin{remark}
	For $c_{1}=1$ in Theorem \ref{thminterphase12}, the distribution of inter-phase times of mixed time-changed Erlang queue reduces to that of fractional Erlang queue (see Ascione {\it et al.} (2020), Theorem 4.2).
\end{remark}
\begin{remark}\label{2cor12}
	Let $Y_{1},Y_{2},\dots, Y_{k}$ be the inter-phase times of mixed time changed Erlang queue and $f_{Y_{1}}(\cdot)$ be the probability density function of $Y_{1}$, that is, $f_{Y_{1}}(t)=\frac{\mathrm{d}}{\mathrm{d}t}\mathrm{Pr}(Y_{1}\leq t)$. It is given by (see Beghin (2012), Eq. (2.38))
	\begin{equation*}
		f_{Y_{1}}(t)=\frac{k\mu}{c_{1}} t^{\alpha_{1}-1}\sum_{r=0}^{\infty}\Big(-\frac{c_{2}}{c_{1}}\Big)^{r}t^{(\alpha_{1}-\alpha_{2})r}E^{r+1}_{\alpha_{1},\alpha_{1}+(\alpha_{1}-\alpha_{2})r}\Big(-\frac{k\mu t^{\alpha_{1}}}{c_{1}}\Big).
	\end{equation*}
	Let $X=Y_{1}+Y_{2}+\dots +Y_{k}$. As $Y_{i}$'s are iid, we have 
	\begin{equation}\label{distX12}
		f_{X}(t)=\frac{\mathrm{d}}{\mathrm{d}t}\mathrm{Pr}(X\leq t)=f_{Y_{1}}^{*(k)}(t).
	\end{equation}
	So, its Laplace transform is (see Beghin (2012), Eq. (2.37))
	\begin{equation}\label{servlp12}
		\tilde{f}_{X}(z)=(\tilde{f}_{Y_{1}}(z))^{k}=\Big(\frac{k\mu}{k\mu+c_{1}z^{\alpha_{1}}+c_{2}z^{\alpha_{2}}}\Big)^{k}.
	\end{equation} 
Hence, the service times of mixed time-changed Erlang queue are iid with the probability density function as given in \eqref{distX12}. 

For $c_{1}=1$ or $c_{2}=1$ in \eqref{servlp12}, it reduces to the Laplace transform of service times of fractional Erlang queue (see Ascione {\it et al.} (2020), Corollary 4.3).
\end{remark}

Recall that the sojourn time of a process is the time spent in a particular state before its transition to another state.

The proof of the following result follows along the similar lines to that of Theorem \ref{thminterarrival}. Thus, it is omitted. Here, we need to work with the state probabilities of queue length process $\{\mathcal{L}^{\alpha_{1},\alpha_{2}}(t)\}_{t\geq0}$. 
\begin{theorem}\label{thmsojourn}
The sojourn times $S^{\alpha_{1},\alpha_{2}}$ of queue length process $\{\mathcal{L}^{\alpha_{1},\alpha_{2}}(t)\}_{t\geq0}$ in a state $n>0$ are iid with the following distribution function:
\begin{align}\label{sojourn12}
	\mathrm{Pr}(S^{\alpha_{1},\alpha_{2}}>t)&=\sum_{r=0}^{\infty}\Big(-\frac{c_{2}}{c_{1}}\Big)^{r}t^{(\alpha_{1}-\alpha_{2})r}E^{r+1}_{\alpha_{1},(\alpha_{1}-\alpha_{2})r+1}\Big(-\frac{(\lambda+k\mu)}{c_{1}}t^{\alpha_{1}}\Big)\nonumber\\
	&\hspace{0.5cm} -\sum_{r=0}^{\infty}\Big(-\frac{c_{2}}{c_{1}}\Big)^{r+1}t^{(\alpha_{1}-\alpha_{2})(r+1)}E^{r+1}_{\alpha_{1},(\alpha_{1}-\alpha_{2})(r+1)+1}\Big(-\frac{(\lambda+k\mu)}{c_{1}}t^{\alpha_{1}}\Big),
\end{align}
where $E_{\alpha,\beta}^{\gamma}(\cdot)$ is three parameter Mittag-Leffler function defined in \eqref{Mittag12}.
\end{theorem}
\begin{remark}
	For $c_{1}=1$ in Theorem \ref{thmsojourn}, the sojourn time of queue length process of mixed time-changed Erlang queue reduces to that of fractional Erlang queue (see Ascione {\it et al.} (2020), Theorem 4.4).
\end{remark}
\begin{remark}
The sojourn time $S^{\alpha_{1},\alpha_{2}}$ of queue length process of mixed time-changed Erlang queue in a state $0$ is the arrival time of first customer in an empty system. So, it is distributed according to \eqref{intrarr12}.
\end{remark}
\begin{remark}
	Recall that the mixed time-changed Erlang queue is a generalization of fractional Erlang queue, for $c_{1}=1$ or $c_{2}=1$ it reduces to fractional Erlang queue. As in fractional Erlang queue inter-arrival and inter-phase times are not independent (see Ascione {\it et al.} (2020), Corollary 4.6),
	the inter-arrival and inter-phase times of mixed time-changed Erlang queue are not necessarily independent. 
\end{remark}
\subsection{Distribution of busy period}
The busy period of queue is defined as a duration of time since the first customer enters the empty system till the system becomes empty again.

\begin{theorem}\label{thmbusy12}
Let $\mathcal{B}^{\alpha_{1},\alpha_{2}}$ be the busy period of mixed time-changed Erlang queue, and let $F_{\mathcal{B}^{\alpha_{1},\alpha_{2}}}(t)=\mathrm{Pr}(\mathcal{B}^{\alpha_{1},\alpha_{2}}\leq t)$ be its distribution function. Then
\begin{equation*}
	F_{\mathcal{B}^{\alpha_{1},\alpha_{2}}}(t)=\sum_{h=0}^{\infty}\frac{k\mu A_{k,h}^{0}}{c_{1}^{a_{k,h}^{0}}}h_{a_{k,h}^{0}}^{*(a_{k,h}^{0})}(t),
\end{equation*}
where $a_{k,h}^{0}$, $A_{k,h}^{0}$ and $h_{a_{k,h}^{0}}(\cdot)$ are as given in \eqref{amr12}, \eqref{Amr12} and \eqref{hN12}, respectively.
\end{theorem}
\begin{proof}
Let us consider a process $\{\mathcal{L}_{*}(t)\}_{t\geq0}$ 
whose state probabilities ${\textbf p}_{n}(t)=\mathrm{Pr}(\mathcal{L}_{*}(t)=n|\mathcal{L}_{*}(0)=k)$, $n\geq0$ solve the following system of differential equations (see Luchak (1958), Eq. (36)-Eq. (38)):
\begin{align*}
	\frac{\mathrm{d}}{\mathrm{d}t}{\textbf p}_{0}(t)&=k\mu {\textbf p}_{1}(t),\\
	\frac{\mathrm{d}}{\mathrm{d}t}{\textbf p}_{n}(t)&=-(\lambda+k\mu){\textbf p}_{n}(t)+k\mu{\textbf p}_{n+1}(t),\,1\leq n \leq k,\\
	\frac{\mathrm{d}}{\mathrm{d}t}{\textbf p}_{n}(t)&=(\lambda+k\mu){\textbf p}_{n}(t)+k\mu{\textbf p}_{n+1}(t)+\lambda{\textbf p}_{n-k}(t),\,n\geq k+1,
\end{align*}
with the initial condition ${\textbf p}_{k}(0)=1$.

Note that the process $\{\mathcal{L}_{*}(t)\}_{t\geq0}$ behaves similar to $\{\mathcal{L}(t)\}_{t\geq0}$ except that it starts from $k$ instead of 0, which indicates that the first customer entered the system, and 0 is its absorbing state. Thus, by construction, the distribution of busy period of Erlang queue is $F_{\mathcal{B}}(t)={\textbf p}_{0}(t)$, where $F_{\mathcal{B}}(t)=\mathrm{Pr}(\mathcal{B}\leq t)$ is given in Luchak (1958). 

Let us define $\mathcal{L}_{*}^{\alpha_{1},\alpha_{2}}(t)\coloneqq\mathcal{L}_{*}(Y_{\alpha_{1},\alpha_{2}}(t))$, $t\geq0$ and denote the $n$th jump time of $\{\mathcal{L}^{\alpha_{1},\alpha_{2}}(t)\}_{t\geq0}$ by $T_{n}$. Then, $\mathcal{L}^{\alpha_{1},\alpha_{2}}(T_{n})$ is time-homogeneous Markov chain. So, $F_{\mathcal{B}^{\alpha_{1},\alpha_{2}}}(t)={\textbf p}_{0}^{\alpha_{1},\alpha_{2}}(t)$. Thus,
\begin{equation}\label{busy12}
	F_{\mathcal{B}^{\alpha_{1},\alpha_{2}}}(t)=\int_{0}^{\infty}F_{\mathcal{B}}(y)\mathrm{Pr}\{Y_{\alpha_{1},\alpha_{2}}(t)\in \mathrm{d}y\}.
\end{equation}
On taking the Laplace transform of \eqref{busy12}, we get
\begin{align}
	\tilde{F}_{\mathcal{B}^{\alpha_{1},\alpha_{2}}}(z)&=(c_{1}z^{\alpha_{1}-1}+c_{2}z^{\alpha_{2}-1})\int_{0}^{\infty}F_{\mathcal{B}}(y)e^{-y(c_{1}z^{\alpha_{1}}+c_{2}z^{\alpha_{2}})}\mathrm{d}y\nonumber\\
	&=(c_{1}z^{\alpha_{1}-1}+c_{2}z^{\alpha_{2}-1})\sum_{h=0}^{\infty}\frac{k\lambda^{h}(k\mu)^{k(h+1)}}{h!(hk+k)!}\int_{0}^{\infty}\int_{0}^{y}u^{k+h(k+1)-1}e^{-(\lambda+k\mu)u}\nonumber\\
	&\hspace{4.2cm}e^{-y(c_{1}z^{\alpha_{1}}+c_{2}z^{\alpha_{2}})}\mathrm{d}u\mathrm{d}y \,\,\,(\text{see Ascione {\it et al.} (2020), Eq. (13)})\nonumber\\
	&=(c_{1}z^{\alpha_{1}-1}+c_{2}z^{\alpha_{2}-1})\sum_{h=0}^{\infty}\frac{k\lambda^{h}(k\mu)^{k(h+1)}}{h!(hk+k)!}\int_{0}^{\infty}u^{k+h(k+1)-1}e^{-(\lambda+k\mu)u}\nonumber\\
	&\hspace{9.2cm}\int_{u}^{\infty}e^{-y(c_{1}z^{\alpha_{1}}+c_{2}z^{\alpha_{2}})}\mathrm{d}y\mathrm{d}u\nonumber\\
	&=\frac{1}{z}\sum_{h=0}^{\infty}\frac{k\lambda^{h}(k\mu)^{k(h+1)}}{h!(hk+k)!}\int_{0}^{\infty}u^{k+h(k+1)-1}e^{-(\lambda+k\mu+c_{1}z^{\alpha_{1}}+c_{2}z^{\alpha_{2}})u}\mathrm{d}u\nonumber\\
	&=\sum_{h=0}^{\infty}\frac{k\lambda^{h}(k\mu)^{k(h+1)}(k+h(k+1)-1)!}{h!(hk+k)!}\frac{z^{-1}}{(\lambda+k\mu+c_{1}z^{\alpha_{1}}+c_{2}z^{\alpha_{2}})^{k+h(k+1)}}.\label{busy121}
\end{align}
On taking the inverse Laplace transform of \eqref{busy121}, we obtain
\begin{equation}\label{busy1212}
	F_{\mathcal{B}^{\alpha_{1},\alpha_{2}}}(t)=\sum_{h=0}^{\infty}\frac{k\mu A_{k,h}^{0}}{c_{1}^{a_{k,h}^{0}}}\mathbb{L}^{-1}\bigg(\bigg(\frac{z^{\frac{-1}{a_{k,h}^{0}}}}{z^{\alpha_{1}}+\frac{c_{2}}{c_{1}}z^{\alpha_{2}}+\frac{\lambda+k\mu}{c_{1}}}\bigg)^{a_{k,h}^{0}};t\bigg).
\end{equation}
finally, on using \eqref{rholp} in \eqref{busy1212}, we get the required result.
\end{proof}
\begin{remark}
	For $c_{1}=1$ in Theorem \ref{thmbusy12}, the distribution of busy period $\mathcal{B}^{\alpha_{1},\alpha_{2}}$ of mixed time-changed Erlang queue reduces to that of fractional Erlang queue (see Ascione {\it et al.} (2020), Theorem 6.3).
\end{remark}
\subsection{Conditional waiting time distribution}
For a single channel queue, the distribution of waiting time is obtained by using the Markov property (see Gaver (1954)). However, it cannot be used to obtain the distribution of waiting time for mixed time-changed Erlang queue because of its semi-Markovian nature. But we can give some information about its waiting time by using some sort of conditioning argument, as done in Ascione {\it et al.} (2020).

The following distribution function will be used: ${\psi^{\alpha_{1},\alpha_{2}}}(t)
=\mathrm{Pr}(\mathcal{P}^{\alpha_{1},\alpha_{2}}>t)$.
It is given in \eqref{intephase12}.
Let $Y$ be a random variable with the following distribution function:
\begin{align}\label{rvY12}
	F_{Y}(t)=\mathrm{Pr}(\mathcal{P}^{\alpha_{1},\alpha_{2}}\leq t_{0}+t|\mathcal{P}^{\alpha_{1},\alpha_{2}}\geq t_{0})=1-\frac{\psi^{\alpha_{1},\alpha_{2}}(t_{0}+t)}{\psi^{\alpha_{1},\alpha_{2}}(t_{0})},
\end{align} 	
with parameters $t_{0}$ and $k\mu$.

On taking the Laplace transform of $f_{\psi}(t)=(t_{0}+t)^{r}$, $r>0$, we obtain
\begin{align}
	\tilde{f}_{\psi}(z)&=\int_{0}^{\infty}e^{-zt}(t_{0}+t)^{r}\mathrm{d}t\nonumber\\
	&=e^{zt_{0}}\int_{t_{0}}^{\infty}e^{-zy}y^{r}\mathrm{d}y\nonumber\\
	&=\frac{e^{zt_{0}}}{z^{r+1}}\Gamma(r+1,zt_{0}),\label{lpgamma12}
\end{align}
where $\Gamma(y,z)$ is the analytic extension of the upper incomplete gamma function on $\mathbb{C}^{2}$ (see Gradshteyn {\it et al.} (2007)).
Also, let $f_{Y}(t)$ be the probability density function of $Y$. Then, \begin{align}\label{densityrvY12}
	f_{Y}(t)=-\frac{1}{\psi^{\alpha_{1},\alpha_{2}}(t_{0})}\frac{\mathrm{d}}{\mathrm{d}t}\psi^{\alpha_{1},\alpha_{2}}(t_{0}+t).
\end{align}
On taking the Laplace transform of  \eqref{densityrvY12} and by using \eqref{lpgamma12}, we get
{\small\begin{align}\label{lpgrml12}
	\tilde{f}_{Y}(z)
	&=1-\Bigg(\sum_{r=0}^{\infty}\sum_{m=0}^{\infty}\Big(-\frac{c_{2}}{c_{1}}\Big)^{r}
	\Big(-\frac{k\mu}{c_{1}}\Big)^{m}
	 \frac{(m+r)!\Gamma((\alpha_{1}-\alpha_{2})r+\alpha_{1}m+1,zt_{0})}{r!m!\Gamma(\alpha_{1}m+(\alpha_{1}-\alpha_{2})r+1)}z^{(\alpha_{2}-\alpha_{1})r-\alpha_{1}m}\nonumber\\
	 &\hspace{1.5cm} -\sum_{r=0}^{\infty}\sum_{m=0}^{\infty}\Big(-\frac{c_{2}}{c_{1}}\Big)^{r+1}
		\Big(-\frac{k\mu}{c_{1}}\Big)^{m}\frac{(m+r)!}{r!m!}\nonumber\\
		& \hspace{1.8cm}\cdot
		\frac{\Gamma((\alpha_{1}-\alpha_{2})(r+1)+\alpha_{1}m+1,zt_{0})}{\Gamma(\alpha_{1}m+(\alpha_{1}-\alpha_{2})(r+1)+1)} z^{(\alpha_{2}-\alpha_{1})(r+1)-\alpha_{1}m}\Bigg)\frac{e^{zt_{0}}}{\psi^{\alpha_{1},\alpha_{2}}(t_{0})},
\end{align}}
where we have used  \eqref{Mittag12} and \eqref{intephase12}.

Let $W^{\alpha_{1},\alpha_{2}}_{t}$ be the time that a customer who enters the system at time $t > 0$ has to spend in the
service. Consider the following random set:
\begin{equation*}
	\mathcal{E}_{t}^{\alpha_{1},\alpha_{2}}({\omega})=\{0\leq s \leq t: \mathcal{L}^{\alpha_{1},\alpha_{2}}(s^{-})(\omega)=\mathcal{L}^{\alpha_{1},\alpha_{2}}(s)(\omega)+1\},
\end{equation*}
which is the set of time instants before $t$ at which a phase has been completed. Now, let us define a random variable as follows:
\begin{equation*}
	\mathcal{N}^{\alpha_{1},\alpha_{2}}_{t}=|\mathcal{E}_{t}^{\alpha_{1},\alpha_{2}}|,
\end{equation*}
which denotes the number of phases completed before time $t$. Also, consider a random variable
\begin{equation*}
	\tau^{\alpha_{1},\alpha_{2}}_{t}=\begin{cases}
	\sup \mathcal{E}_{t}^{\alpha_{1},\alpha_{2}},\, \mathcal{N}^{\alpha_{1},\alpha_{2}}_{t}>0,\\
	t,\, \mathcal{N}^{\alpha_{1},\alpha_{2}}_{t}=0,
	\end{cases}
\end{equation*}
that is, the last time instant before $t$ at which a phase has been completed. For fixed $t$, $\tau^{\alpha_{1},\alpha_{2}}_{t}$ is a Markov time whose values are the jump times of $\{\mathcal{S}^{\alpha_{1},\alpha_{2}}(t)\}_{t\geq0}$ except when $\tau^{\alpha_{1},\alpha_{2}}_{t}=t$. As $\{\mathcal{S}^{\alpha_{1},\alpha_{2}}(t)\}_{t\geq0}$ is a semi-Markov process, its semi-regenerative set contains all the points of discontinuities, and given that $\tau^{\alpha_{1},\alpha_{2}}_{t}=t_{0}$ for some $t_{0}<t$, the past and future of $\{\mathcal{S}^{\alpha_{1},\alpha_{2}}(t)\}_{t\geq0}$ are independent (see Cinlar (1974)). Thus, for waiting time, we condition on $\{\mathcal{L}^{\alpha_{1},\alpha_{2}}(t)=n\}$ as well as on $\{\tau^{\alpha_{1},\alpha_{2}}_{t}=t_{0}\}$. Let us denote
\begin{equation*}
	w^{\alpha_{1},\alpha_{2}}(x;t,t_{0},n)\,\mathrm{d}x=\mathrm{Pr}(W_{t}^{\alpha_{1},\alpha_{2}}\in \mathrm{d}x|\tau^{\alpha_{1},\alpha_{2}}_{t}=t_{0},\mathcal{L}^{\alpha_{1},\alpha_{2}}(t)=n).
\end{equation*} 	
A customer has to wait for the completion of $n$ independent phases (since we have conditioned on  $\{\mathcal{L}^{\alpha_{1},\alpha_{2}}(t)=n\}$), one of the $n$ phases started at $t_{0}$ that has not been completed till time $t$. That is, 
\begin{equation*}
	\mathbb{E}(W^{\alpha_{1},\alpha_{2}}_{t}|\tau^{\alpha_{1},\alpha_{2}}_{t}=t_{0},\mathcal{L}^{\alpha_{1},\alpha_{2}}(t)=n)=\sum_{j=1}^{n}T_{j},
\end{equation*}
where $T_{j}$'s are $n$ independent random variables. Here, for $1\leq j\leq n-1$, $T_{j}$ denotes the duration of $j$th phase whose service is not yet started and  $T_{n}$ denotes the remaining time  duration of the phase which is in service at time $t$.
Note that $T_{j}$, $1\leq j\leq n-1$ are iid from Theorem \ref{thminterphase12}. Thus, $W=\sum_{j=1}^{n-1}T_{j}$ has the distribution given by \eqref{distX12}. For the distribution of $T_{n}$, we know that the last phase was completed at $t_{0}$ since we have conditioned on $\tau^{\alpha_{1},\alpha_{2}}_{t}=t_{0}$. So, the phase being served at time $t$ started its service at time $t_{0}$, thus $T_{n}$ is distributed as \eqref{rvY12} with parameters $t-t_{0}$ and $k\mu$. Then, we have
\begin{equation}\label{wtden12}
	w^{\alpha_{1},\alpha_{2}}(x;t,t_{0},n)=f_{{W}}*f_{T_{n}}(x).
\end{equation}
On taking the Laplace transform of \eqref{wtden12}, and using \eqref{servlp12} and \eqref{lpgrml12}, we get
\begin{align*}
	\tilde{w}(z;t,t_{0},n)
	&=\Big(\frac{k\mu}{k\mu+c_{1}z^{\alpha_{1}}+c_{2}z^{\alpha_{2}}}\Big)^{n-1}\Bigg(1-\frac{e^{z(t-t_{0})}}{\psi^{\alpha_{1},\alpha_{2}}(t-t_{0})}\Bigg(\sum_{r=0}^{\infty}\sum_{m=0}^{\infty}\Big(-\frac{c_{2}}{c_{1}}\Big)^{r}\Big(-\frac{k\mu}{c_{1}}\Big)^{m}\nonumber\\
	&\ \ \cdot\frac{(m+r)!\Gamma((\alpha_{1}-\alpha_{2})r+\alpha_{1}m+1,z(t-t_{0}))}{r!m!\Gamma(\alpha_{1}m+(\alpha_{1}-\alpha_{2})r+1)}z^{(\alpha_{2}-\alpha_{1})r-\alpha_{1}m}\nonumber\\
	&\ \ -\sum_{r=0}^{\infty}\sum_{m=0}^{\infty}\Big(-\frac{c_{2}}{c_{1}}\Big)^{r+1}
		\Big(-\frac{k\mu}{c_{1}}\Big)^{m}\frac{(m+r)!}{r!m!}\nonumber\\
		&\ \ \cdot\frac{\Gamma((\alpha_{1}-\alpha_{2})(r+1)+\alpha_{1}m+1,z(t-t_{0}))}{\Gamma(\alpha_{1}m+(\alpha_{1}-\alpha_{2})(r+1)+1)}z^{(\alpha_{2}-\alpha_{1})(r+1)-\alpha_{1}m}\Bigg)\Bigg).
\end{align*} 
\subsection{Sample paths simulation}
Here, we give plots of sample path simulations of the Mixed time-changed Erlang queue. Following lemmas will be used for this simulations.
\begin{lemma}\label{lemma1}
	Let $X$ be an exponential random variable with rate $\lambda$ and $\{D_{\alpha_{1},\alpha_{2}}(t)\}_{t\geq0}$ be mixed stable subordinator which is independent of $X$. Then, $D_{\alpha_{1},\alpha_{2}}(X)\overset{d}{=} T^{\alpha_{1},\alpha_{2}}$, where $T^{\alpha_{1},\alpha_{2}}$ is the inter-arrival time of the mixed time-changed Erlang queue.
\end{lemma}
\begin{proof}
Consider the following:
\begin{align}\label{1stlema12}
	\mathrm{Pr}\{D_{\alpha_{1},\alpha_{2}}(X)>t\}&=\mathrm{Pr}\{X>Y_{\alpha_{1},\alpha_{2}}(t)\}\nonumber\\
	&=\int_{0}^{\infty}\mathrm{Pr}\{X>y\}\mathrm{Pr}\{Y_{\alpha_{1},\alpha_{2}}(t)\in \mathrm{d}y\}\nonumber\\
	&=\int_{0}^{\infty}e^{-\lambda y}\mathrm{Pr}\{Y_{\alpha_{1},\alpha_{2}}(t)\in \mathrm{d}y\}.
\end{align}
On taking the Laplace transform of \eqref{1stlema12} and by using \eqref{mixedlp12}, we obtain
\begin{equation*}\label{1stlemmalp12}
	\mathbb{L}(\mathrm{Pr}\{D_{\alpha_{1},\alpha_{2}}(X)>t\})(z)=\frac{c_{1}z^{\alpha_{1}-1}+c_{2}z^{\alpha_{2}-1}}{\lambda+c_{1}z^{\alpha_{1}}+c_{2}z^{\alpha_{2}}},
\end{equation*}
whose inversion yields the required result.
\end{proof}

\begin{lemma}\label{lemma2}
Let $X$ be a non-negative random variable independent of the mixed stable subordinator $\{D_{\alpha_{1},\alpha_{2}}(t)\}_{t\geq0}$. Then, 
$D_{\alpha_{1},\alpha_{2}}(X)\overset{d}{=}c_{1}^{1/\alpha_{1}}X^{1/\alpha_{1}}D_{\alpha_{1}}(1)+c_{2}^{1/\alpha_{2}}X^{1/\alpha_{2}}D_{\alpha_{2}}(1),$
where $D_{\alpha_{1}}(1)$ and $D_{\alpha_{2}}(1)$ are independent random variables.
\end{lemma}
\begin{proof}
Note that,
\begin{align*}
	\mathrm{Pr}\{D_{\alpha_{1},\alpha_{2}}(X)>t\}
	&=\int_{0}^{\infty}\mathrm{Pr}\{D_{\alpha_{1},\alpha_{2}}(y)>t\}\mathrm{Pr}\{X\in\mathrm{d}y\}\nonumber\\
	&=\int_{0}^{\infty}\mathrm{Pr}\{c_{1}^{1/\alpha_{1}}D_{\alpha_{1}}(y)+c_{2}^{1/\alpha_{2}}D_{\alpha_{2}}(y)>t\}\mathrm{Pr}\{X\in\mathrm{d}y\}\nonumber\\
	&=\int_{0}^{\infty}\mathrm{Pr}\{c_{1}^{1/\alpha_{1}}y^{1/\alpha_{1}}D_{\alpha_{1}}(1)+c_{2}^{1/\alpha_{2}}y^{1/\alpha_{2}}D_{\alpha_{2}}(1)>t\}\mathrm{Pr}\{X\in\mathrm{d}y\}\nonumber\\
	&=\mathrm{Pr}\{c_{1}^{1/\alpha_{1}}X^{1/\alpha_{1}}D_{\alpha_{1}}(1)+c_{2}^{1/\alpha_{2}}X^{1/\alpha_{2}}D_{\alpha_{2}}(1)>t\},
\end{align*}
where we have used $D_{\alpha_{1},\alpha_{2}}(t)\overset{d}{=}c_{1}^{1/\alpha_{1}}D_{\alpha_{1}}(t)+c_{2}^{1/\alpha_{2}}D_{\alpha_{2}}(t)$, $t\geq0$ (see Aletti {\it et al.} (2018), p. 704) such that $\{D_{\alpha_{1}}(t)\}_{t\geq0}$ and $\{D_{\alpha_{2}}(t)\}_{t\geq0}$ are independent stable subordinators.
This proves the required result.
\end{proof}
\begin{figure}
	\centering
	\includegraphics[width=1\textwidth]{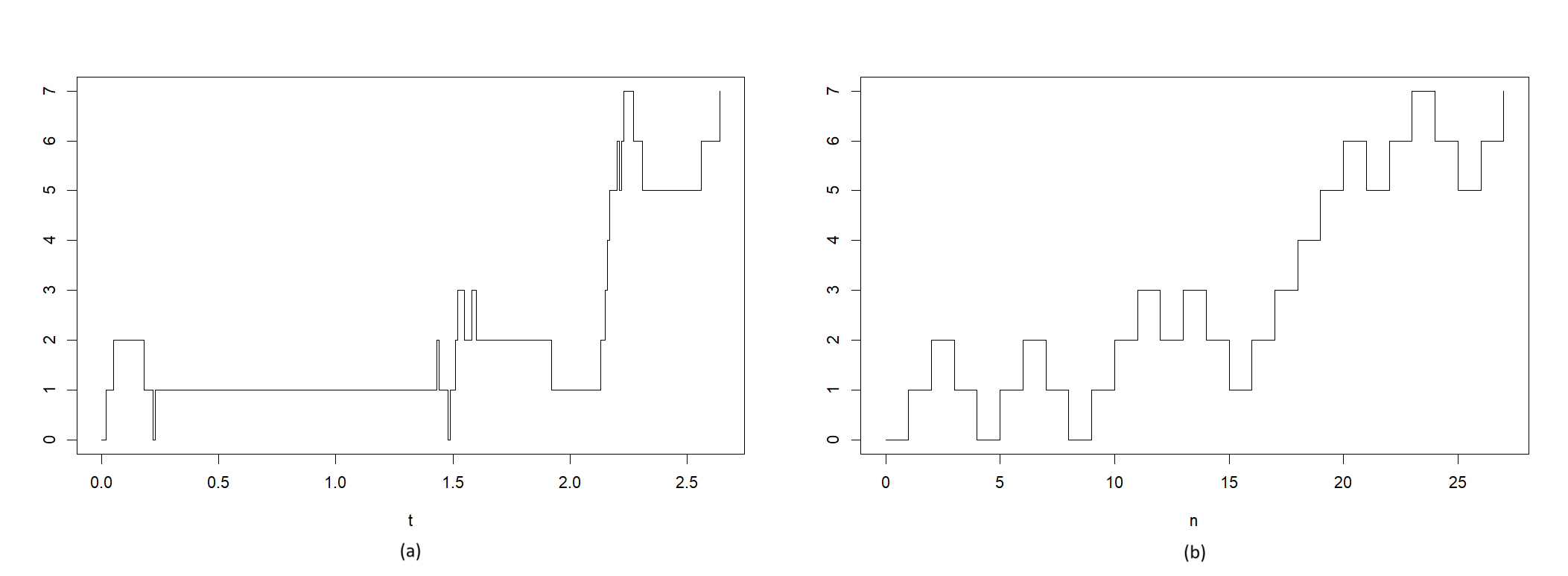}
	\caption{{\small Plot (a) represents the sample path simulation of $N^{\alpha_{1},\alpha_{2}}(t)$ and Plot (b) represents the sample path simulation of $N^{\alpha_{1},\alpha_{2}}(T_{n})$ such that $T_{n}$ are the jump times of $N^{\alpha_{1},\alpha_{2}}(t)$ for parameters $c_1=0.4$, $c_2=0.6$, $\alpha_{1}=0.5$, $\alpha_{2}=0.3$, $k=4$, $\lambda=6$ and $\mu=5$.}}
	\label{fig1} 
\end{figure}
By using Lemma \ref{lemma1} and Lemma \ref{lemma2}, we have the following result:
\begin{proposition}\label{corollary12}
Let $X$ be an exponential random variable with rate $\lambda>0$, $\{D_{\alpha_{1},\alpha_{2}}(t)\}_{t\geq0}$ be mixed stable subordinator independent of $X$, and Y be a random variable such that $Y\overset{d}{=}T^{\alpha_{1},\alpha_{2}}$. Then,
\begin{equation*}
	Y\overset{d}{=}c_{1}^{1/\alpha_{1}}X^{1/\alpha_{1}}D_{\alpha_{1}}(1)+c_{2}^{1/\alpha_{2}}X^{1/\alpha_{2}}D_{\alpha_{2}}(1),
\end{equation*}
where $D_{\alpha_{1}}(1)$ and $D_{\alpha_{2}}(1)$ are independent.
\end{proposition}

We use the modified version of Gillespie algorithm (see Cohoy {\it et al.} (2015)) for simulating the mixed time-changed Erlang queue.
It is based on the fact that $\mathcal{Q}^{\alpha_{1},\alpha_{2}}(T_{n})$ is again a Markov chain with the same transition probabilities as that of Erlang queue, where $T_{n}$ is the $n$th jump times of $\{\mathcal{Q}^{\alpha_{1},\alpha_{2}}(t)\}_{t\geq0}$. This algorithm is used by Ascione {\it et al.} (2020) for simulating the fractional Erlang queue.  The algorithm to simulate the mixed time-changed Erlang queue is similar 
to the one used for simulating fractional Erlang queue. However, we use a different distribution for a random variable $I$ given in Ascione {\it et al.} (2020). If $\mathcal{Q}^{\alpha_{1},\alpha_{2}}(T_{n})=(0,0)$ then we use Proposition \ref{corollary12} to simulate $I\overset{d}{=}T^{\alpha_{1},\alpha_{2}}$, and if $\mathcal{Q}^{\alpha_{1},\alpha_{2}}(T_{n})\in\mathcal{H}$ then we simulate $I\overset{d}{=}S^{\alpha_{1},\alpha_{2}}$, where the distributions of $T^{\alpha_{1},\alpha_{2}}$ and $S^{\alpha_{1},\alpha_{2}}$ are given in \eqref{intrarr12} and \eqref{sojourn12}, respectively. 

\section*{Acknowledgement}
The research of first author was supported by a UGC fellowship, NTA reference no. 231610158041, Govt. of India.


\begin{thebibliography}{00}
\bibitem{Aletti2018}
Aletti, G., Leonenko, N. and Merzbach, E. (2018). Fractional Poisson fields. \textit{J. Stat. Phys.} \textbf{170}(4), 700–730.
				
\bibitem{Applebaum2009}
Applebaum, D. (2009). L\'evy Processes and Stochastic Calculus. Cambridge University Press, Cambridge.
				
				
\bibitem{Ascione2018}
Ascione, G., Leonenko, N. and Pirozzi E. (2018). Fractional queues with catastrophes and their transient behaviour. \textit{Math.} \textbf{6}(9), 159.
				
\bibitem{Ascione2020}
Ascione, G., Leonenko, N. and Pirozzi E. (2020). Fractional Erlang queues.
\textit{Stochastic Process. Appl.}
\textbf{130}(6), 3249-3276.
				
\bibitem{Beghin2012}
Beghin., L. (2012). Random-time processes governed by differential equations of fractional distributed order.
\textit{Chaos Solitons Fractals.}
\textbf{45}(11), 1314--1327.
				
\bibitem{Beghin2009}
Beghin, L. and Orsingher, E. (2009). Fractional Poisson processes and related planar random motions.
\textit{Electron. J. Probab.} \textbf{14}(61), 1790-1827.
				
%
				
\bibitem{Cahoy2015}
Cahoy, D. O., Polito, F. and Phoha, V. (2015). Transient behavior of fractional queues and related processes. \textit{Methodol. Comput. Appl. Probab.} \textbf{17}(3), 739-759.
						
\bibitem{Cinlar1974}
Cinlar, E. (1974). Markov additive process and semi-regeneration. {\it Discussion Paper No. 118}, North
western University.
						
\bibitem{DiCrescenzo2003}
Di Crescenzo, A., Giorno, V., Kumar, B. K., Nobile, A. G. (2003). On the $M/M/1$ queue with catastrophes and its continuous approximation. \textit{Queueing Syst.} \textbf{43}(4), 329-347. 
				
%
%
%
				
\bibitem{Gaver1954}
Gaver, D. P. (1954). The influence of servicing times in queuing processes. \textit{J. Operations Res. Soc. Amer.} \textbf{2}(2), 139-149.
				
\bibitem{Gihman1975}
Gikhman, I. I. and Skorokhod, A. V. (1975). The Theory of Stochastic Processes. {II}. 218.
				
\bibitem{Giorno2018}
Giorno, V., Nobile, A. G. and Pirozzi,E. (2018). A state-dependent queueing system with asymptotic logarithmic distribution. \textit{J. Math. Anal. Appl. }\textbf{458}(2) 949-966.
				
				
\bibitem{Gradshteyn2007}
Gradshteyn., I. S., Ryzhik., I. M. Jeffrey., A. and Zwillinger., D. (2007). Table of Integrals, Series, and Products (Seventh Edition). Academic Press, Boston.
				
\bibitem{Griffiths2006}
Griffiths, J. D., Leonenko, G. M. and Williams, J.E. (2006). The transient solution to $M/E_{k}/1$ queue.
\textit{Oper. Res. Lett.} \textbf{34}, 349-354.
				
\bibitem{Haubold2011}
Haubold, H. J., Mathai, A. M. and Saxena, R. K. (2011). Mittag-Leffler functions and their applications.
\textit{J. Appl. Math.} 298628, 51 pp.
				
\bibitem{Kataria2021}
Kataria, K. K. and Khandakar, M. (2021). Mixed fractional risk process. \textit{J. Math. Anal. Appl.} \textbf{504}(1), Article 125379.
				
%
\bibitem{Kataria2025}
Kataria, K. K. and Vishwakarma, P. (2025). On time-changed linear birth–death–immigration process. \textit{J. Theoret. Probab.} \textbf{38}(1), 21.
%
				
\bibitem{Kilbas2006}
Kilbas, A. A., Srivastava, H. M. and Trujillo, J. J. (2006). Theory and Applications of Fractional Differential Equations. Elsevier Science B.V., Amsterdam.
				
\bibitem{Luchak1956}
Luchak, G. (1956). The solution of the single-channel queuing equations characterized by a time-dependent Poisson-distributed arrival rate and a general class of holding times. \textit{Operations Res.} \textbf{4}(6), 711-732.
	 
				
\bibitem{Luchak1958}
Luchak, G. (1958). The continuous time solution of the equations of the single channel queue with a general  class of service-time distributions by the method of generating functions.
\textit{J. Roy. Statist. Soc. Ser. B.} \textbf{20}(1), 176-181.
				
				
\bibitem{Meerschaert2011}
Meerschaert, M. M., Nane, E., Vellaisamy, P. (2011).  The fractional Poisson process and the inverse stable subordinator.
\textit{ Electron. J. Probab.} \textbf{16}(59), 1600-1620.
				
\bibitem{Meerschaert2013}
Meerschaert, M. M.  and Straka, P. (2013). Inverse stable subordinators.
\textit{Math. Model. Nat. Phenom.} \textbf{8}(2), 1-16.
				
\bibitem{Orsingher2011}
Orsingher, E. and Polito, F. (2011). On a fractional linear birth-death process. {\it Bernoulli.} {\bf17}(1), 114-137.
%
\bibitem{Orsingher2012}
Orsingher, E. and Polito, F. (2012). The space-fractional Poisson process. {\it Statist. Probab. Lett.} {\bf82}(4), 852-858.
				
\bibitem{Spath2006}
Spath, D. and F{\"a}hnrich, K. P. (2006). Advances in Services Innovations. Springer, Berlin, Heidelberg.
%
%
				
				
\end{thebibliography}
\end{document}